\documentclass[a4paper, intlimits, reqno]{amsart}

\usepackage[english]{babel}
\usepackage[T1]{fontenc}
\usepackage[utf8]{inputenc}

\usepackage{amsmath}
\usepackage{amssymb}
\usepackage{MnSymbol}
\usepackage{amsthm}
\allowdisplaybreaks
\usepackage{amsfonts}
\usepackage{mathrsfs} 
\usepackage{mathtools}
\usepackage{nicefrac}
\usepackage{enumitem}
\usepackage{multicol}
\usepackage{url}
\usepackage{dsfont}
\usepackage{color}
\usepackage[numbers]{natbib}
\usepackage{prettyref}

\newrefformat{def}{Definition \ref{#1}}
\newrefformat{rem}{Remark \ref{#1}}
\newrefformat{sect}{Section \ref{#1}}
\newrefformat{sub}{Section \ref{#1}}
\newrefformat{prop}{Proposition \ref{#1}}
\newrefformat{thm}{Theorem \ref{#1}}
\newrefformat{lem}{Lemma \ref{#1}}
\newrefformat{chap}{Chapter \ref{#1}}
\newrefformat{cor}{Corollary \ref{#1}}
\newrefformat{par}{Paragraph \ref{#1}}
\newrefformat{ex}{Example \ref{#1}}
\newrefformat{part}{Part \ref{#1}}
\newrefformat{fig}{Figure \ref{#1}}
\newrefformat{app}{Appendix \ref{#1}}
\newrefformat{que}{Question \ref{#1}}
\newrefformat{cond}{Condition \ref{#1}}
\newrefformat{conv}{Convention \ref{#1}}

\swapnumbers
\newtheoremstyle{dotless}{}{}{\itshape}{}{\bfseries}{}{}{}
\theoremstyle{dotless}
\theoremstyle{plain}
\newtheorem{thm}{Theorem}[section]
\newtheorem{lem}[thm]{Lemma}
\newtheorem{prop}[thm]{Proposition}
\newtheorem{cor}[thm]{Corollary}
\theoremstyle{definition}
\newtheorem{defn}[thm]{Definition}
\newtheorem{rem}[thm]{Remark}

\newtheorem{exa}[thm]{Example}

\providecommand{\differential}{\mathrm{d}}
\renewcommand{\d}{\differential}

\newcommand{\N} {\mathbb{N}}

\newcommand{\R} {\mathbb{R}}

\begin{document}

\title[Approximation property]{The approximation property for weighted spaces of differentiable functions}
\author[K.~Kruse]{Karsten Kruse}
\address{TU Hamburg \\ Institut f\"ur Mathematik \\
Am Schwarzenberg-Campus~3 \\
Geb\"aude E \\
21073 Hamburg \\
Germany}
\email{karsten.kruse@tuhh.de}



\date{\today}
\begin{abstract}
 We study spaces $\mathcal{CV}^{k}(\Omega,E)$ of $k$-times continuously partially differentiable functions 
 on an open set $\Omega\subset\mathbb{R}^{d}$ with values in a locally convex Hausdorff space $E$.
 The space $\mathcal{CV}^{k}(\Omega,E)$ is given a weighted topology generated by a family of weights $\mathcal{V}^{k}$. 
 For the space $\mathcal{CV}^{k}(\Omega,E)$ and its subspace $\mathcal{CV}^{k}_{0}(\Omega,E)$ of functions that 
 vanish at infinity in the weighted topology we try to answer the question whether their elements can be approximated 
 by functions with values in a finite dimensional subspace. We derive sufficient conditions for an affirmative answer to 
 this question using the theory of tensor products.
\end{abstract}
\maketitle
\section{Introduction}
This paper is dedicated to the following problem. Which vector-valued $k$-times continuously partially differentiable 
functions can be approximated in a weighted topology by functions with values in a finite dimensional subspace? 
The answer to this question is closely related to the theory of tensor products and the so-called approximation property. 
A locally convex Hausdorff space $X$ is said to have (Schwartz') approximation property 
if the identity $I_{X}$ on $X$ is contained in the closure of $\mathfrak{F}(X)$ in $L_{\kappa}(X)$ 
where $L_{\kappa}(X)$ denotes the space of continuous linear operators from $X$ to $X$ 
equipped with the topology of uniform convergence on the absolutely convex compact subsets of $X$
and $\mathfrak{F}(X)$ its subspace of operators with finite rank.

The case $k=0$ is well-studied. In \cite{B3}, \cite{B1} and \cite{B2} Bierstedt considered the space $\mathcal{CV}(\Omega,E)$ 
of all continuous functions $f\colon\Omega\to E$ from a completely regular Hausdorff space $\Omega$ to 
a locally convex Hausdorff space $(E,(p_{\alpha})_{\alpha\in\mathfrak{A}})$ over a field $\mathbb{K}$ with a topology induced by a 
Nachbin-family $\mathcal{V}:=(\nu_{j})_{j\in J}$ of weights, i.e.\ 
the space 
\[
  \mathcal{CV}(\Omega,E):=\{f\in\mathcal{C}(\Omega,E)\;|\;\forall\;j\in J,\,\alpha\in\mathfrak{A}:
  |f|_{j,\alpha}<\infty\}
\]
 where $\mathcal{C}(\Omega,E):=\mathcal{C}^{0}(\Omega,E)$ is the space of continuous functions from $\Omega$ to $E$ and
\[
  |f|_{j,\alpha}:=\sup_{x\in\Omega}p_{\alpha}(f(x))\nu_{j}(x). 
\]
Recall that a family $\mathcal{V}:=(\nu_{j})_{j\in J}$ of non-negative functions $\nu_{j}\colon\Omega\to[0,\infty)$ 
is called a Nachbin-family of weights if the functions $\nu_{j}$ are upper semi-continuous and the family is directed, 
i.e.\ for every $j,i\in J$ there are $k\in J$ and $C>0$ such that $\max(\nu_{i},\nu_{j})\leq C\nu_{k}$. 
The notion $\mathcal{U}\leq\mathcal{V}$ for two Nachbin-families means that for every $\mu\in\mathcal{U}$ there 
is $\nu\in\mathcal{V}$ such that $\mu\leq \nu$.

From the perspective of our problem the space $\mathcal{CV}(\Omega,E)$ has an interesting topological subspace, namely, 
the space $\mathcal{CV}_{0}(\Omega,E)$ consisting of the functions that vanish at infinity when weighted which is given by 
\begin{align*}
  \mathcal{CV}_{0}(\Omega,E):=\{f\in\mathcal{CV}(\Omega,E)
  \;|\;&\forall\;\varepsilon>0,\,j\in J,\,\alpha\in\mathfrak{A}
  \;\exists\;K\subset \Omega\;\text{compact}:\;|f|_{\Omega\setminus K,j,\alpha}<\varepsilon\} 
\end{align*}
 where 
\[
  |f|_{\Omega\setminus K,j,\alpha}:=\sup_{x\in\Omega\setminus K}p_{\alpha}(f(x))\nu_{j}(x).
\]
One of the main results from \cite{B1} solves our problem for $k=0$, Nachbin-families of weights 
and involves $k_{\mathbb{R}}$-spaces. A completely regular space $\Omega$ is a $k_{\mathbb{R}}$-space if for any 
completely regular space $Y$ and any map $f\colon \Omega \to Y$, 
whose restriction to each compact $K\subset\Omega$ is continuous, the map is already continuous on $\Omega$ 
(see \cite[(2.3.7) Proposition, p.\ 22]{buchwalter}). Obviously, every locally compact Hausdorff space 
is a $k_{\mathbb{R}}$-space. Further examples of $k_{\mathbb{R}}$-spaces are metrisable spaces 
by \cite[Proposition 11.5, p.\ 181]{james} and \cite[3.3.20, 3.3.21 Theorem, p.\ 152]{engelking} as well as 
strong duals of Fr\'{e}chet-Montel spaces by \cite[Proposition 3.27, p.\ 95]{fabian} and \cite[4.11 Theorem, p.\ 39]{kriegl}.
 
\begin{thm}[{\cite[5.5 Theorem, p.\ 205-206]{B1}}]\label{thm:bierstedt}
Let $E$ be a locally convex Hausdorff space, $\Omega$ a completely regular Hausdorff space 
and $\mathcal{V}$ a Nachbin-family on $\Omega$ such that 
one of the following conditions is satisfied.
\begin{enumerate}
\item [(i)] $\mathcal{Z}:=\left\{v\colon \Omega\to \mathbb{R}\;|\; v\;\text{constant},\;v\geq 0\right\}\leq \mathcal{V}$.
\item [(ii)] $\mathcal{W}:=\left\{\mu \chi_{K}\;|\; \mu>0,\;K\subset\Omega\;\text{compact}\right\}\leq \mathcal{V}$, 
where $\chi_{K}\colon \Omega\to \mathbb{R}$ is the characteristic function of $K$, and $\Omega$ is a $k_{\mathbb{R}}$-space.
\end{enumerate}
Then the following holds.
\begin{enumerate}
\item [a)] $\mathcal{CV}_{0}(\Omega)\otimes E$ is dense in $\mathcal{CV}_{0}(\Omega,E)$.
\item [b)] If $E$ is complete, then 
\[
 \mathcal{CV}_{0}(\Omega,E)\cong \mathcal{CV}_{0}(\Omega)\varepsilon E
 \cong \mathcal{CV}_{0}(\Omega)\widehat{\otimes}_{\varepsilon} E.
\]
\item [c)] $\mathcal{CV}_{0}(\Omega)$ has the approximation property.
\end{enumerate}
\end{thm}

Here $\mathcal{CV}_{0}(\Omega)\otimes E$ stands for the tensor product, $\mathcal{CV}_{0}(\Omega)\widehat{\otimes}_{\varepsilon} E$ 
for the completion of the injective tensor product and 
$\mathcal{CV}_{0}(\Omega)\varepsilon E:=L_{e}(\mathcal{CV}_{0}(\Omega)_{\kappa}',E)$ 
for the $\varepsilon$-product of Schwartz of the spaces $\mathcal{CV}_{0}(\Omega):=\mathcal{CV}_{0}(\Omega,\mathbb{K})$ and $E$. 
Part a) gives an affirmative answer to our question for the space $\mathcal{CV}_{0}(\Omega,E)$ since it implies 
that for every $\varepsilon>0$, $\alpha\in\mathfrak{A}$, $j\in J$ and $f\in\mathcal{CV}_{0}(\Omega,E)$ there are 
$m\in\mathbb{N}$, $f_{n}\in\mathcal{CV}_{0}(\Omega)$ and $e_{n}\in E$, $1\leq n\leq m$, such that
\[
  \bigl|f-\sum_{n=1}^{m}f_{n}e_{n}\bigr|_{j,\alpha}<\varepsilon.
\]
Concerning $\mathcal{CV}(\Omega,E)$, the answer to our question is not that satisfying but still affirmative if we make some 
restrictions on $E$. If $E$ has the approximation property, 
then $E\otimes_{\varepsilon}\mathcal{CV}(\Omega)$ is dense in $E\varepsilon\mathcal{CV}(\Omega)$. Due to the symmetries 
$\mathcal{CV}(\Omega)\otimes_{\varepsilon}E\cong E\otimes_{\varepsilon}\mathcal{CV}(\Omega)$ and 
$\mathcal{CV}(\Omega)\varepsilon E\cong E\varepsilon\mathcal{CV}(\Omega)$, 
we get that $\mathcal{CV}(\Omega)\otimes_{\varepsilon}E$ is dense in 
$\mathcal{CV}(\Omega)\varepsilon E\cong\mathcal{CV}(\Omega,E)$ if $E$ 
is a semi-Montel space with approximation property and $\mathcal{Z}\leq \mathcal{V}$ or $\Omega$ is a $k_{\mathbb{R}}$-space 
by \cite[2.12 Satz (1), p.\ 141]{B2}. A second condition for an affirmative answer without 
supposing that $E$ has the approximation property 
but putting more restrictions on $\mathcal{CV}(\Omega)$ can be found in \cite[2.12 Satz (2), p.\ 141]{B2}.

We aim to prove a version of Bierstedt's theorem for spaces of weighted continuously partially differentiable functions. 
To the best of our knowledge the approximation problem was not considered in a general setting for $k>0$ 
and open $\Omega\subset\mathbb{R}^{d}$, i.e.\ 
to derive sufficient conditions on the weights and the spaces such that the answer is positive.
For special cases with $\Omega=\mathbb{R}^{d}$ like the Schwartz space an affirmative answer was already given in 
e.g.\ \cite[Proposition 9, p.\ 108]{Schwartz1955} and \cite[Th\'{e}or\`{e}me 1, p.\ 111]{Schwartz1955}.  
For the space of $k$-times continuously partially differentiable functions on open $\Omega\subset\mathbb{R}^{d}$ 
with the topology of uniform convergence of all partial derivatives up to order $k$ on compact sets 
a positive answer can be found in e.g.\ \cite[Proposition 44.2, p.\ 448]{Treves} and \cite[Theorem 44.1, p.\ 449]{Treves}.
Let us consider the latter space for a moment and the corresponding proof given by Tr\`{e}ves in \cite{Treves}. 
The space $\mathcal{C}^{k}(\Omega,E)$ of $k$-times continuously 
partially differentiable functions on a locally compact Hausdorff space $\Omega$ if $k=0$ 
resp.\ open $\Omega\subset\mathbb{R}^{d}$ if $k\in\mathbb{N}\cup\{\infty\}$ 
is equipped with the system of seminorms given by
\begin{equation}\label{cv-complete.0}
q_{K,l,\alpha}(f):=\sup_{\substack{x\in K\\\beta\in\mathbb{N}^{d}_{0},|\beta|\leq l}}
p_{\alpha}\bigl(\partial^{\beta}f(x)\bigr),\quad 
f\in \mathcal{C}^{k}(\Omega,E),
\end{equation}
for $K\subset\Omega$ compact, $l\in\mathbb{N}_{0}$, $0\leq l\leq k$ if $k<\infty$, and $\alpha\in \mathfrak{A}$. 
For $E=\mathbb{K}$ we fix the notion $\mathcal{C}^{k}(\Omega):=\mathcal{C}^{k}(\Omega,\mathbb{K})$ 
and denote by $\mathcal{C}^{k}_{c}(\Omega)$ the space of all functions in $\mathcal{C}^{k}(\Omega)$ 
having compact support.
Tr\`{e}ves' affirmative answer to our question has the following form.

\begin{thm}[{\cite[Proposition 44.2, p.\ 448]{Treves} and 
\cite[Theorem 44.1, p.\ 449]{Treves}}]\label{thm:treves}
Let $E$ be a locally convex Hausdorff space, $k\in\mathbb{N}_{0}\cup\{\infty\}$ 
and $\Omega$ a locally compact Hausdorff space if $k=0$ resp.\ 
an open subset of $\mathbb{R}^{d}$ if $k>0$. 
Then the following is true.
\begin{enumerate}
\item [a)] $\mathcal{C}^{0}_{c}(\Omega)\otimes E$ is dense in $\mathcal{C}^{0}(\Omega,E)$.
\item [b)] $\mathcal{C}^{\infty}_{c}(\Omega)\otimes E$ is dense in $\mathcal{C}^{k}(\Omega,E)$.
\item [c)] If $E$ is complete, then 
\[
 \mathcal{C}^{k}(\Omega,E)\cong \mathcal{C}^{k}(\Omega)\widehat{\otimes}_{\varepsilon} E.
\]
\end{enumerate}
\end{thm}

We observe that $\mathcal{CW}(\Omega,E)=\mathcal{CW}_{0}(\Omega,E)=\mathcal{C}^{0}(\Omega,E)$ 
equipped with the usual topology of uniform convergence on compact subsets of $\Omega$ which means that 
\prettyref{thm:bierstedt} contains the case $k=0$ of the preceding theorem since locally compact Hausdorff spaces are 
$k_{\mathbb{R}}$-spaces. 
The proofs of \prettyref{thm:bierstedt} a) and \prettyref{thm:treves} a) are done by using 
different partitions of unity, the first uses the partition of unity from \cite[23, Lemma 2, p.\ 71]{nachbin1965} and the second the one 
from \cite[Chap.\ IX, \S4.3, Corollary, p.\ 186]{bourbakiII}. The key idea for the proof of \prettyref{thm:treves} b) is 
an approximation in three steps relying on part a) and convolution. 
First, for every $f\in\mathcal{C}^{k}(\Omega,E)$ there is an approximation 
$\widetilde{f}\in\mathcal{C}^{k}_{c}(\Omega,E)$ of $f$ by multiplication of $f$ with a suitable cut-off function. 
Second, for every $\widetilde{f}\in\mathcal{C}^{k}_{c}(\Omega,E)$ the 
convolution $\widetilde{f}\ast\rho_{n}$ of $\widetilde{f}$ with a sequence $(\rho_{n})$ of molifiers 
in $\mathcal{C}^{\infty}_{c}(\Omega)$ 
converges to $\widetilde{f}$ in $\mathcal{C}^{k}(\Omega,\widehat{E})$ where $\widehat{E}$ denotes the completion of $E$ 
(approximation by regularisation). 
Third, for every $\widetilde{f}\in\mathcal{C}^{k}_{c}(\Omega,E)$ 
there is an approximation $g\in \mathcal{C}^{0}_{c}(\Omega)\otimes E$ in the topology of $\mathcal{C}^{0}(\Omega,E)$ by part a). 
Using the properties of the convolution, one gets that $g\ast\rho_{n}\in\mathcal{C}^{\infty}_{c}(\Omega)\otimes E$ 
and approximates $\widetilde{f}\ast\rho_{n}$ for $n$ large enough in $\mathcal{C}^{k}(\Omega,\widehat{E})$ 
which itself is identical to the completion of $\mathcal{C}^{k}(\Omega,E)$.

The outline of our paper is along the lines of Tr\`{e}ves' proof. After introducing some notation and preliminaries in Section 2, 
we define the weighted spaces $\mathcal{CV}^{k}(\Omega,E)$ and $\mathcal{CV}^{k}_{0}(\Omega,E)$ in Section 3 and show that 
they are complete if the family of weights $\mathcal{V}^{k}$ is locally bounded away from zero 
(see \prettyref{def:loc.bound.away}). Then we treat their relation to the space $\mathcal{C}^{k}_{c}(\Omega,E)$
of functions in $\mathcal{C}^{k}(\Omega,E)$ with compact support where the condition of local boundedness 
of a family of weights comes into play (see \prettyref{def:loc.bound}). 
We formulate a cut-off criterion (see \prettyref{def:cut-off}) which is a sufficient condition for the density of
$\mathcal{C}^{k}_{c}(\Omega,E)$ in $\mathcal{CV}^{k}_{0}(\Omega,E)$ for locally bounded $\mathcal{V}^{k}$. 
We close the third section with the relation between tensor products and our problem on finite dimensional approximation.
In Section 4 we define the convolution $f\ast g$ of $f\in\mathcal{C}^{k}(\R^{d},E)$ and $g\in\mathcal{C}^{n}(\R^{d})$ 
when one of them is compactly supported and prove an approximation by regularisation result. 
In the last section we verify the corresponding part a) of \prettyref{thm:treves} for $\mathcal{CV}^{0}_{0}(\Omega,E)$ 
with locally compact $\Omega$ where we adapt the proof of \prettyref{thm:bierstedt} a) in a way that we can 
use the partition of unity from \cite[Chap.\ IX, \S4.3, Corollary, p.\ 186]{bourbakiII} instead and weaken
the condition of upper semi-continuity of the weights to being locally bounded and locally bounded away from zero. 
Then we mix all ingredients to get our main \prettyref{thm:weighted.diff} which is a version of 
\prettyref{thm:bierstedt} and \ref{thm:treves} for barrelled $\mathcal{CV}^{k}_{0}(\Omega)$ 
with a family of weights $\mathcal{V}^{k}$ being locally bounded and locally bounded away from zero if 
$\mathcal{CV}^{k}_{0}(\Omega,E)$ fulfils the cut-off criterion.
\section{Notation and Preliminaries}
We set $\mathbb{N}_{\infty}:=\mathbb{N}\cup\{\infty\}$ and $\mathbb{N}_{0,\infty}:=\mathbb{N}_{0}\cup\{\infty\}$. 
For $k\in\mathbb{N}_{0,\infty}$ we use the notation
$\langle k \rangle:=\{n\in\mathbb{N}_{0}\;|\; 0\leq n\leq k\}$ if $k\neq\infty$ 
and $\langle k \rangle:=\mathbb{N}_{0}$ if $k=\infty$.
We equip the spaces $\mathbb{R}^{d}$, $d\in\mathbb{N}$, and $\mathbb{C}$ with the usual Euclidean norm $|\cdot|$, 
write $\overline{M}$ for the closure of a subset $M\subset\mathbb{R}^{d}$ and
denote by $\mathbb{B}_{r}(x):=\{w\in\mathbb{R}^{d}\;|\;|w-x|<r\}$ 
the ball around $x\in\mathbb{R}^{d}$ with radius $r>0$.

By $E$ we always denote a non-trivial locally convex Hausdorff space, in short lcHs, over the field 
$\mathbb{K}=\mathbb{R}$ or $\mathbb{C}$ equipped with a directed fundamental system of 
seminorms $(p_{\alpha})_{\alpha\in \mathfrak{A}}$. 
If $E=\mathbb{K}$, then we set $(p_{\alpha})_{\alpha\in \mathfrak{A}}:=\{|\cdot|\}.$ 
Further, we denote by $\widehat{E}$ the completion of a locally convex Hausdorff space $E$.
For details on the theory of locally convex spaces see \cite{F/W/Buch}, \cite{Jarchow} or \cite{meisevogt1997}.

A function $f\colon\Omega\to E$ on an open set $\Omega\subset\mathbb{R}^{d}$ to a locally convex Hausdorff space $E$ is called 
continuously partially differentiable ($f$ is $\mathcal{C}^{1}$) 
if for the $n$-th unit vector $e_{n}\in\mathbb{R}^{d}$ the limit
\[
(\partial^{e_{n}})f(x):=(\partial^{e_{n}})^{E}f(x):=(\partial_{x_{n}})^{E}f(x)
:=\lim_{\substack{h\to 0\\ h\in\mathbb{R}, h\neq 0}}\frac{f(x+he_{n})-f(x)}{h}
\]
exists in $E$ for every $x\in\Omega$ and $\partial^{e_{n}}f$ 
is continuous on $\Omega$ ($\partial^{e_{n}}f$ is $\mathcal{C}^{0}$) for every $1\leq n\leq d$.
For $k\in\mathbb{N}$ a function $f$ is said to be $k$-times continuously partially differentiable 
($f$ is $\mathcal{C}^{k}$) if $f$ is $\mathcal{C}^{1}$ and all its first partial derivatives are $\mathcal{C}^{k-1}$.
A function $f$ is called infinitely continuously partially differentiable ($f$ is $\mathcal{C}^{\infty}$) 
if $f$ is $\mathcal{C}^{k}$ for every $k\in\mathbb{N}$.
For $k\in\mathbb{N}_{\infty}$ the linear space of all functions $f\colon\Omega\to E$ which are $\mathcal{C}^{k}$ 
is denoted by $\mathcal{C}^{k}(\Omega,E)$. Its subspace of functions with compact support 
is written as $\mathcal{C}^{k}_{c}(\Omega,E)$ where we denote the support 
of $f\in\mathcal{C}^{k}(\Omega,E)$ by $\operatorname{supp} f$.

Let $f\in\mathcal{C}^{k}(\Omega,E)$. For $\beta\in\mathbb{N}_{0}^{d}$ with 
$|\beta|:=\sum_{n=1}^{d}\beta_{n}\leq k$ we set 
$\partial^{\beta_{n}}f:=(\partial^{\beta_{n}})^{E}f:=f$ if $\beta_{n}=0$, and
\[
\partial^{\beta_{n}}f:=(\partial^{\beta_{n}})^{E}f
:=\underbrace{(\partial^{e_{n}})^{E}\cdots(\partial^{e_{n}})^{E}}_{\beta_{n}\text{-times}}f
\]
if $\beta_{n}\neq 0$ as well as 
\[
\partial^{\beta}f:=(\partial^{\beta})^{E}f
:=\partial^{\beta_{1}}\cdots\partial^{\beta_{d}}f.
\]
Due to the vector-valued version of Schwarz' theorem $\partial^{\beta}f$ is independent of the order of the partial 
derivatives on the right-hand side and we call $|\beta|$ the order of differentiation.
Further, we observe that $e'\circ f\in \mathcal{C}^{k}\left(\Omega\right)$ and 
$(\partial^{\beta})^{\mathbb{K}}(e'\circ f)=e'\circ (\partial^{\beta})^{E}f$ for every $e'\in E'$,
$f\in\mathcal{C}^{k}(\Omega,E)$ and $|\beta|\leq k$.

By $L(F,E)$ we denote the space of continuous linear operators from $F$ to $E$ where $F$ and $E$ are locally convex Hausdorff spaces. 
If $E=\mathbb{K}$, we just write $F':=L(F,\mathbb{K})$ for the dual space.
If $F$ and $E$ are (linearly topologically) isomorphic, we write $F\cong E$.
The so-called $\varepsilon$-product of Schwartz is defined by 
\begin{equation}\label{notation0}
F\varepsilon E:=L_{e}(F_{\kappa}',E)
\end{equation}
where $F'$ is equipped with the topology of uniform convergence on absolutely convex compact subsets of $F$ and 
$L(F_{\kappa}',E)$ is equipped with the topology of uniform convergence on equicontinuous subsets of $F'$ 
(see \cite[Chap.\ I, \S1, D\'{e}finition, p.\ 18]{Sch1}). 
It is symmetric which means that $F\varepsilon E\cong E\varepsilon F$ and in the literature the definition of the 
$\varepsilon$-product is sometimes done the other way around, i.e.\ $E\varepsilon F$ is defined by the right-hand side 
of \eqref{notation0}.
We write $F\,\widehat{\otimes}_{\varepsilon}E$ for the completion of the injective tensor product 
$F\otimes_{\varepsilon}E$ and denote by $\mathfrak{F}(E)$ the space of linear operators from $E$ to $E$ with finite rank. 
We recall from the introduction that a locally convex Hausdorff space $E$ is said to have (Schwartz') approximation property 
if the identity $I_{E}$ on $E$ is contained in the closure of $\mathfrak{F}(E)$ in $L_{\kappa}(E):=L_{\kappa}(E,E)$ 
which is equipped with the topology of uniform convergence on the absolutely convex compact subsets of $E$. 
The space $E$ has the approximation property if and only if $E\otimes F$ is dense in $E\varepsilon F$ 
for every locally convex Hausdorff space (every Banach space) $F$ by \cite[Satz 10.17, p.\ 250]{Kaballo}.
For more information on the theory of $\varepsilon$-products 
and tensor products see \cite{Defant}, \cite{Jarchow} and \cite{Kaballo}. 
\section{Weighted vector-valued differentiable functions and the $\varepsilon$-product}

In this section we introduce the spaces $\mathcal{CV}^{k}(\Omega,E)$ and 
$\mathcal{CV}^{k}_{0}(\Omega,E)$ we want to consider. Then we turn to the question of completeness 
of $\mathcal{CV}^{k}(\Omega,E)$ and $\mathcal{CV}^{k}_{0}(\Omega,E)$ and when $\mathcal{C}^{k}_{c}(\Omega,E)$ 
is dense in the latter space. At the end of this section we describe their connection 
to the $\varepsilon$-product and the (completion of the) injective tensor product 
and derive sufficient conditions such that they coincide.

\begin{defn}[{weight}]
Let $k\in\mathbb{N}_{0,\infty}$.  We say that 
$\mathcal{V}^{k}:=(\nu_{j,l})_{j\in J,l\in\langle k \rangle}$ 
is a (directed) family of weights on a locally compact Hausdorff space $\Omega$ if $\nu_{j,l}\colon \Omega\to [0,\infty)$ for every 
$j\in J$, $l\in\langle k\rangle$ and
\[
  \forall\;j_{1},j_{2}\in J,\,l_{1},l_{2}\in\langle k\rangle\;\exists\;j_{3}\in J,\,l_{3}\in\langle k\rangle,\,C>0\;
  \forall\;i\in\{1,2\}:\;
  \nu_{j_{i},l_{i}}\leq C\nu_{j_{3},l_{3}}
\]
as well as
\[
  \forall\;l\in \langle k\rangle,\,x\in\Omega\;\exists\;j\in J:\;0<\nu_{j,l}(x).
\]
\end{defn}

\begin{defn}\label{def:weighted_diff_spaces}
 For $k\in\mathbb{N}_{0,\infty}$ and a (directed) family
 $\mathcal{V}^{k}:=(\nu_{j,l})_{j\in J,l\in\langle k \rangle}$ 
 of weights on a locally compact Hausdorff space $\Omega$ if $k=0$ or an open set $\Omega\subset\mathbb{R}^{d}$ if 
 $k\in\mathbb{N}_{\infty}$ we define the space of weighted continuous resp.\
 $k$-times continuously partially differentiable functions with values in an lcHs $E$ as
 \[
  \mathcal{CV}^{k}(\Omega,E):=\{f\in\mathcal{C}^{k}(\Omega,E)\;|\;\forall\;j\in J,\,l\in\langle k \rangle,\,
  \alpha\in\mathfrak{A}:\;|f|_{j,l,\alpha}<\infty\} 
 \]
 where 
  \[
  |f|_{j,l,\alpha}:=\sup_{\substack{x\in\Omega\\ \beta\in\mathbb{N}_{0}^{d},|\beta|\leq l}}
  p_{\alpha}\bigl((\partial^{\beta})^{E}f(x)\bigr)\nu_{j,l}(x).
  \]
  We define the topological subspace of $\mathcal{CV}^{k}(\Omega,E)$ consisting of the functions 
  that vanish with all their derivatives when weighted at infinity by 
  \begin{align*}
  \mathcal{CV}^{k}_{0}(\Omega,E):=\{f\in\mathcal{CV}^{k}(\Omega,E)\;|\;&\forall\;j\in J,\,
  l\in\langle k \rangle,\,\alpha\in\mathfrak{A},\,\varepsilon>0\\
  &\exists\;K\subset \Omega\;\text{compact}:\;|f|_{\Omega\setminus K,j,l,\alpha}<\varepsilon\} 
  \end{align*}
   where 
  \[
  |f|_{\Omega\setminus K,j,l,\alpha}:=\sup_{\substack{x\in\Omega\setminus K\\ \beta\in\mathbb{N}_{0}^{d},|\beta|\leq l}}
  p_{\alpha}\bigl((\partial^{\beta})^{E}f(x)\bigr)\nu_{j,l}(x).
  \]
\end{defn}

It is easily seen that these spaces are locally convex Hausdorff spaces with a directed system of seminorms 
due to our assumptions on the family $\mathcal{V}^{k}$ of weights. 

\begin{rem}
Suppose that in the definition of the space $\mathcal{CV}^{k}(\Omega,E)$ the weights also depend on $\beta\in\N_{0}^{d}$, i.e.\ 
the seminorms used to define $\mathcal{CV}^{k}(\Omega,E)$ are of the form 
\[
|f|_{j,l,\alpha}^{\sim}:=\sup_{\substack{x\in\Omega\\ \beta\in\mathbb{N}_{0}^{d},|\beta|\leq l}}
  p_{\alpha}\bigl((\partial^{\beta})^{E}f(x)\bigr)\nu_{j,l,\beta}(x).
\]
Without loss of generality we may always use weights which are independent of $\beta$. Namely, 
by setting $\nu_{j,l}:=\max_{\beta\in\N_{0}^{d},|\beta|\leq l}\nu_{j,l,\beta}$ for $j\in J$ and $l\in\langle k \rangle$, 
we can switch to the usual system of seminorms $(|f|_{j,l,\alpha})$ induced by the weights $(\nu_{j,l})$ 
which is equivalent to $(|f|_{j,l,\alpha}^{\sim})$. 
\end{rem}

The standard structure of a directed family $\mathcal{V}^{k}$ of weights 
on a locally compact Hausdorff space $\Omega$ is given by the following. 
Let $(\Omega_{j})_{j\in J}$ be a family of sets such that 
$\Omega_{j}\subset\Omega_{j+1}$ for all $j\in J$ with $\Omega=\bigcup_{j\in J}\Omega_{j}$.
Let $\widetilde{\nu}_{j,l}\colon\Omega\to (0,\infty)$ be continuous for all $j\in J$ 
and $l\in\langle k \rangle$, increasing in $j\in J$, i.e.\ 
$\widetilde{\nu}_{j,l}\leq \widetilde{\nu}_{j+1,l}$, and in $l\in\langle k \rangle$, i.e.\
$\widetilde{\nu}_{j,l}\leq \widetilde{\nu}_{j,l+1}$ if $l+1\in\langle k \rangle$,
such that 
\[
\nu_{j,l}(x)=\chi_{\Omega_{j}}(x)\widetilde{\nu}_{j,l}(x),\quad x\in \Omega,
\]
for every $j\in J$ and $l\in\langle k \rangle$ where $\chi_{\Omega_{j}}$ is the indicator function of $\Omega_{j}$. 
Further, we remark that the spaces $\mathcal{CV}^{k}(\Omega,E)$ and $\mathcal{CV}^{k}_{0}(\Omega,E)$ might coincide 
which is already mentioned in \cite[1.3 Bemerkung, p.\ 189]{B1} for $k=0$. 

\begin{rem}\label{rem:cv=cv_0}
If for every $j\in J$ and $l\in\langle k \rangle$, there are $i\in J$ and 
$m\in\langle k \rangle$ such that for all $\varepsilon>0$ there is a compact set $K\subset\Omega$ with 
$\nu_{j,l}(x)\leq \varepsilon\nu_{i,m}(x)$ for all $x\in\Omega\setminus K$, 
then $\mathcal{CV}^{k}(\Omega,E)=\mathcal{CV}^{k}_{0}(\Omega,E)$.
\end{rem}

Examples of spaces where this happens are $\mathcal{C}^{k}(\Omega,E)$ with the topology of uniform convergence of 
all partial derivatives up to order $k$ on compact subsets of $\Omega$ and the Schwartz space $\mathcal{S}(\mathbb{R}^{d},E)$.

\begin{exa}\label{ex:standard_diff_spaces}
Let $E$ be an lcHs, $k\in\mathbb{N}_{0,\infty}$ and $\Omega\subset\mathbb{R}^{d}$ open. Then we have 
\begin{enumerate}
\item [a)] $\mathcal{C}^{k}(\Omega,E)=\mathcal{CW}^{k}(\Omega,E)=\mathcal{CW}^{k}_{0}(\Omega,E)$
with $\mathcal{W}^{k}:=\{\nu_{j,l}:=\chi_{\Omega_{j}}\;|\;j\in\mathbb{N},\,l\in\langle k \rangle\}$ 
where $(\Omega_{j})_{j\in\mathbb{N}}$ is a compact exhaustion of $\Omega$, 
\item [b)] $\mathcal{S}(\mathbb{R}^{d},E)=\mathcal{CV}^{\infty}(\mathbb{R}^{d},E)=\mathcal{CV}^{\infty}_{0}(\mathbb{R}^{d},E)$ with 
$\mathcal{V}^{\infty}:=\{\nu_{j,l}\;|\;j\in\mathbb{N},\,l\in\mathbb{N}_{0}\}$ where
$\nu_{j,l}(x):=(1+|x|^{2})^{l/2}$ for $x\in\mathbb{R}^{d}$.
\end{enumerate}
\end{exa}
\begin{proof}
\begin{enumerate}
\item [a)] $(\Omega_{j})_{j\in\mathbb{N}}$ being a compact exhaustion of $\Omega$ means that
$\Omega=\bigcup_{j\in \mathbb{N}}\Omega_{j}$, $\Omega_{j}$ is compact and $\Omega_{j}\subset\mathring{\Omega}_{j+1}$ for all 
$j\in\mathbb{N}$ where $\mathring{\Omega}_{j+1}$ is the set of inner points of $\Omega_{j+1}$.
For compact $\Omega_{j}\subset\Omega$ and $l\in\langle k \rangle$ our claim follows 
from \prettyref{rem:cv=cv_0} with the choice $i:=j$, $m:=l$ and $K:=\Omega_{j}$.
\item [b)] We recall that the Schwartz space is defined by
\[
\mathcal{S}(\mathbb{R}^{d},E)
:=\bigl\{f\in \mathcal{C}^{\infty}(\mathbb{R}^{d},E)\;|\;
\forall\;l\in\mathbb{N}_{0},\;\alpha\in \mathfrak{A}:\;\|f\|_{l,\alpha}<\infty\bigr\}
\]
where 
\[
 \|f\|_{l,\alpha}:=\sup_{\substack{x\in\mathbb{R}^{d}\\ \beta\in\mathbb{N}_{0}^{d},|\beta|\leq l}}
 p_{\alpha}\bigl((\partial^{\beta})^{E}f(x)\bigr)(1+|x|^{2})^{l/2}.
\]
Thus $\mathcal{S}(\mathbb{R}^{d},E)=\mathcal{CV}^{\infty}(\mathbb{R}^{d},E)$. 
We note that for every $j\in\mathbb{N}$, $l\in\mathbb{N}_{0}$ and $\varepsilon>0$ there is $r>0$ such 
that
\[
\frac{\nu_{j,l}(x)}{\nu_{j,2(l+1)}(x)}=\frac{(1+|x|^{2})^{l/2}}{(1+|x|^{2})^{l+1}}=(1+|x|^{2})^{-(l/2)-1}<\varepsilon
\]
for all $x\notin\overline{\mathbb{B}_{r}(0)}=:K$ 
yielding $\mathcal{S}(\mathbb{R}^{d},E)=\mathcal{CV}^{\infty}_{0}(\mathbb{R}^{d},E)$ by \prettyref{rem:cv=cv_0}.
\end{enumerate}
\end{proof}

The question of finite dimensional approximation from the introduction is closely connected to the 
property of a family of weights being locally bounded away from zero. 

\begin{defn}[{locally bounded away from zero}]\label{def:loc.bound.away}
Let $\Omega$ be a locally compact Hausdorff space and $k\in\mathbb{N}_{0,\infty}$. 
A family of weights $\mathcal{V}^{k}$ is called locally bounded away from zero on $\Omega$ if
\[
 \forall\;K\subset\Omega\;\text{compact},\,l\in\langle k \rangle\;\exists\; 
 j\in J:\;\inf_{x\in K}\nu_{j,l}(x)>0.
\]
%\[
% \forall\;K\subset\Omega\;\text{compact},l\in\langle k \rangle\;\exists\; 
% j\in J,\,\beta\in\mathbb{N}_{0}^{d},\,|\beta|\leq l:\;\inf_{x\in K}\nu_{j,l,\beta}(x)>0.
%\]
\end{defn} 

For $k=0$ (and locally compact Hausdorff $\Omega$) this coincides with condition (ii) of \prettyref{thm:bierstedt}. 
It even guarantees that the spaces $\mathcal{CV}^{k}(\Omega,E)$ and $\mathcal{CV}^{k}_{0}(\Omega,E)$ are complete 
for complete $E$.

\begin{prop}\label{prop:cv-complete}
Let $E$ be a complete lcHs, $k\in\mathbb{N}_{0,\infty}$ and $\mathcal{V}^{k}$ be a family of weights 
which is locally bounded away from zero on a locally compact Hausdorff space $\Omega$ ($k=0$) or 
an open set $\Omega\subset\mathbb{R}^{d}$ ($k>0$).
Then $\mathcal{CV}^{k}(\Omega,E)$ and $\mathcal{CV}^{k}_{0}(\Omega,E)$ are complete 
locally convex Hausdorff spaces. 
In particular, they are Fr\'{e}chet spaces if $E$ is a Fr\'{e}chet space and $J$ countable.
\end{prop}
\begin{proof}
Let $(f_{\tau})_{\tau\in\mathcal{T}}$ be a Cauchy net in 
$\mathcal{CV}^{k}(\Omega,E)$. The space $\mathcal{C}^{k}(\Omega,E)$ 
equipped with the usual system of seminorms $(q_{K,l,\alpha})$ given in 
\eqref{cv-complete.0} is complete by \cite[Proposition 44.1, p.\ 446]{Treves}. 
Let $K\subset\Omega$ compact, $l\in\langle k \rangle$ and $\alpha\in \mathfrak{A}$. 
Then there is $j\in J$ such that 
\[
q_{K,l,\alpha}(f)\leq \sup_{x \in K}\nu_{j,l}(x)^{-1}|f |_{j,l,\alpha}
=\bigl(\inf_{x \in K}\nu_{j,l}(x)\bigr)^{-1}
|f|_{j,l,\alpha},\quad f\in\mathcal{CV}^{k}(\Omega, E),
\]
since $\mathcal{V}^{k}$ is locally bounded away from zero implying that
the inclusion $\mathcal{CV}^{k}(\Omega, E)\hookrightarrow\mathcal{C}^{k}(\Omega,E)$ is continuous. 
Thus $(f_{\tau})$ is a Cauchy net in 
$\mathcal{C}^{k}(\Omega,E)$ as well and has a limit $f$ in this space due to the completeness. 
Let $j\in J$, $l\in\langle k \rangle$, $\alpha\in \mathfrak{A}$ and $\varepsilon>0$. 
As this convergence implies pointwise convergence, we have that
for all $x\in \Omega$ and $\beta\in \mathbb{N}_{0}^{d}$, $|\beta|\leq l$,
there exists $\tau_{j,l,\beta, x}\in\mathcal{T}$ such that for all $\tau\geq\tau_{j,l,\beta,x}$
\begin{equation}\label{cv-complete.1}
p_{\alpha}\bigl((\partial^{\beta})^{E}f_{\tau}(x)-(\partial^{\beta})^{E}f(x)\bigr)<\frac{\varepsilon}{2\nu_{j,l}(x)}
\end{equation}	 
if $\nu_{j,l}(x)>0$.
Furthermore, there exists $\tau_{0}\in\mathcal{T}$ such that for all $\tau,\mu\geq \tau_{0}$
\begin{equation}\label{cv-complete.2}
|f_{\tau}-f_{\mu}|_{j,l,\alpha}<\frac{\varepsilon}{2}
\end{equation}
by assumption. Hence we get for all $\tau\geq\tau_{0}$ by choosing $\mu\geq\tau_{j,l,\beta,x},\tau_{0}$
\begin{flalign*}
&\quad\: p_{\alpha}\bigl((\partial^{\beta})^{E}f(x)\bigr)\nu_{j,l}(x)
-p_{\alpha}\bigl((\partial^{\beta})^{E}f_{\tau}(x)\bigr)\nu_{j,l}(x)\\
&\leq p_{\alpha}\bigl((\partial^{\beta})^{E}f_{\tau}(x)-(\partial^{\beta})^{E}f(x)\bigr)\nu_{j,l}(x)\\
&\leq p_{\alpha}\bigl((\partial^{\beta})^{E}f_{\tau}(x)-(\partial^{\beta})^{E}f_{\mu}(x)\bigr)\nu_{j,l}(x)+
p_{\alpha}\bigl((\partial^{\beta})^{E}f_{\mu}(x)-(\partial^{\beta})^{E}f(x)\bigr)\nu_{j,l}(x)\\
&\;{\mathop{<}\limits_{\mathclap{\eqref{cv-complete.1}}}}\; \sup_{z\in \Omega}
p_{\alpha}\bigl((\partial^{\beta})^{E}f_{\tau}(z)-(\partial^{\beta})^{E}f_{\mu}(z)\bigr)
\nu_{j,l}(z)+\frac{\varepsilon}{2} \\
& \leq \sup_{\substack{z \in \Omega\\ \gamma \in \mathbb{N}_{0}^{d}, \, |\gamma| \leq l}}
p_{\alpha}\bigl((\partial^{\gamma})^{E}f_{\tau}(z)-(\partial^{\gamma})^{E}f_{\mu}(z)\bigr)\nu_{j,l}(z)
+\frac{\varepsilon}{2}\\
&=|f_{\tau}-f_{\mu}|_{j,l,\alpha}+\frac{\varepsilon}{2}\\
&\;{\mathop{<}\limits_{\mathclap{\eqref{cv-complete.2}}}}\;\varepsilon
\end{flalign*}
if $\nu_{j,l}(x)>0$. We deduce that for all $\tau\geq\tau_{0}$
\begin{flalign*}
&\quad\: p_{\alpha}\bigl((\partial^{\beta})^{E}f(x)\bigr)\nu_{j,l}(x)
-p_{\alpha}\bigl((\partial^{\beta})^{E}f_{\tau}(x)\bigr)\nu_{j,l}(x)\\
&\leq p_{\alpha}\bigl((\partial^{\beta})^{E}f_{\tau}(x)-(\partial^{\beta})^{E}f(x)\bigr)\nu_{j,l}(x)
<\varepsilon
\end{flalign*}
if $\nu_{j,l}(x)>0$. If $\nu_{j,l}(x)=0$, then this estimate is also fulfilled 
and so $|f_{\tau}-f|_{j,l,\alpha}\leq\varepsilon$ 
as well as $|f|_{j,l,\alpha}\leq\varepsilon+|f_{\tau}|_{j,l,\alpha}$ for all $\tau\geq \tau_{0}$. 
This means that $f\in\mathcal{CV}^{k}(\Omega, E)$ and that $(f_{\tau})$ converges to $f$ 
in $\mathcal{CV}^{k}(\Omega, E)$. Therefore $\mathcal{CV}^{k}(\Omega, E)$ is complete 
and $\mathcal{CV}^{k}_{0}(\Omega,E)$ as well because it is a closed 
subspace of the complete space $\mathcal{CV}^{k}(\Omega,E)$.
\end{proof}

For $k\in\mathbb{N}_{0,\infty}$ and locally compact Hausdorff $\Omega$ ($k=0$) 
or open $\Omega\subset\mathbb{R}^{d}$ ($k>0$) 
we define $\mathcal{CV}^{k}_{c}(\Omega,E)$ to be the subspace of $\mathcal{CV}^{k}(\Omega,E)$ of 
functions with compact support. Obviously we 
have $\mathcal{CV}^{k}_{c}(\Omega,E)\subset\mathcal{CV}^{k}_{0}(\Omega,E)$ 
and $\mathcal{CV}^{k}_{c}(\Omega,E)\subset\mathcal{C}^{k}_{c}(\Omega,E)$.  
On the other hand, the space $\mathcal{C}^{k}_{c}(\Omega,E)$ is a linear subspace of $\mathcal{CV}^{k}_{c}(\Omega,E)$ 
if the family of weights $\mathcal{V}^{k}$ fulfils the definition of local boundedness. 

\begin{defn}[{locally bounded}]\label{def:loc.bound}
Let $\Omega$ be a locally compact Hausdorff space and $k\in\mathbb{N}_{0,\infty}$. 
A family of weights $\mathcal{V}^{k}$ is called locally bounded on $\Omega$ if
\[
 \forall\;K\subset\Omega\;\text{compact},\,j\in J,\,l\in\langle k \rangle:\;
 \sup_{x\in K}\nu_{j,l}(x)<\infty.
\]
\end{defn}

Indeed, if $f\in\mathcal{C}^{k}_{c}(\Omega,E)$, then we have for $K:=\operatorname{supp}f$ 
\begin{align*}
|f|_{j,l,\alpha}&=\sup_{\substack{x\in K\\ \beta\in\mathbb{N}_{0}^{d},|\beta|\leq l}}
p_{\alpha}\bigl((\partial^{\beta})^{E}f(x)\bigr)\nu_{j,l}(x)\\
&\leq\bigl(\sup_{\substack{z\in K\\ \beta\in\mathbb{N}_{0}^{d},|\beta|\leq l}}
p_{\alpha}\bigl((\partial^{\beta})^{E}f(z)\bigr)\bigr)
\sup_{x\in K}\nu_{j,l}(x)
\end{align*}
for all $j\in J$, $l\in\langle k \rangle$ and $\alpha\in\mathfrak{A}$. Hence we have:

\begin{rem}\label{rem:compact.supp}
Let $E$ be an lcHs and  $k\in\mathbb{N}_{0,\infty}$. If $\mathcal{V}^{k}$ is a family of locally bounded weights, 
then $\mathcal{C}^{k}_{c}(\Omega,E)=\mathcal{CV}^{k}_{c}(\Omega,E)$ algebraically.
\end{rem}

Next, we formulate a sufficient criterion for the density of $\mathcal{C}^{k}_{c}(\Omega,E)$ 
in $\mathcal{CV}^{k}_{0}(\Omega,E)$ for $k\in\mathbb{N}_{0,\infty}$, $\Omega\subset\mathbb{R}^{d}$ open 
and locally bounded $\mathcal{V}^{k}$. 

\begin{defn}[{cut-off criterion}]\label{def:cut-off}
Let $E$ be an lcHs, $k\in\mathbb{N}_{0,\infty}$, $\Omega\subset\mathbb{R}^{d}$ open 
and $\mathcal{V}^{k}$ be a family of weights on $\Omega$. We say that $\mathcal{CV}^{k}_{0}(\Omega,E)$ satisfies 
the cut-off criterion if 
\begin{align*}
&\forall\;f\in \mathcal{CV}^{k}_{0}(\Omega,E),\,j\in J,\,l\in\langle k \rangle,\,\alpha\in\mathfrak{A}\;\exists\;
\delta>0\;\forall\;\varepsilon>0\;\exists\;K\subset \Omega\;\text{compact}:\\
&\phantom{\forall\;f\in \mathcal{CV}^{k}_{0}(\Omega,E),\,}
\bigl(K+\overline{\mathbb{B}_{\delta}(0)}\,\bigr)\subset \Omega\quad \text{and}\quad 
|f|_{\Omega\setminus K,j,l,\alpha}<\varepsilon.
\end{align*}
\end{defn}

\begin{rem}\label{rem:whole.cut-off}
If $\Omega=\mathbb{R}^{d}$, then the cut-off criterion is satisfied for any $\delta>0$.
\end{rem}

\begin{exa}\label{ex:cut-off}
Let $E$ be an lcHs, $k\in\mathbb{N}_{0,\infty}$ and $\Omega\subset\mathbb{R}^{d}$ open. 
The space $\mathcal{C}^{k}(\Omega,E)$ with the usual topology of uniform convergence of 
all partial derivatives up to order $k$ on compact subsets of $\Omega$ and the Schwartz space 
$\mathcal{S}(\mathbb{R}^{d},E)$ fulfil the cut-off criterion.
\end{exa}
\begin{proof}
For the Schwartz space this follows directly from \prettyref{ex:standard_diff_spaces} b) and 
\prettyref{rem:whole.cut-off}. 
By \prettyref{ex:standard_diff_spaces} a) we have $\mathcal{C}^{k}(\Omega,E)=\mathcal{CW}^{k}_{0}(\Omega,E)$ with
$\mathcal{W}^{k}:=\{\nu_{j,l}:=\chi_{\Omega_{j}}\;|\;j\in\mathbb{N},\,l\in\langle k \rangle\}$ 
where $(\Omega_{j})_{j\in\mathbb{N}}$ is a compact exhaustion of $\Omega$.
Choosing $K:=\Omega_{j}$ and 
$\delta:=\inf\{|z-x|\;|\;z\in\partial\Omega_{j},\,x\in\partial\Omega_{j+1}\}>0$ for $j\in\mathbb{N}$, 
we note that the cut-off criterion is fulfilled.
\end{proof}

The proof of the density given below uses cut-off functions and the additional $\delta>0$ independent of $\varepsilon>0$ 
allows us to choose a suitable cut-off function whose derivatives can be estimated independently of $\varepsilon$.
But first we recall the following definitions since we need the product rule. Let $\gamma,\beta\in\mathbb{N}^{d}_{0}$. 
We write $\gamma\leq\beta$ if $\gamma_{n}\leq\beta_{n}$ for all $1\leq n\leq d$, and define
\[
\dbinom{\beta}{\gamma}:=\prod_{n=1}^{d}\dbinom{\beta_{n}}{\gamma_{n}}
\]
if $\gamma\leq\beta$ where the right-hand side is defined by ordinary binomial coefficients. 
Now, we can phrase the product rule whose proof follows by induction (just adapt the proof for scalar-valued functions).

\begin{prop}[{product rule}]
Let $E$ be an lcHs, $k\in\mathbb{N}_{0,\infty}$, $\Omega\subset\mathbb{R}^{d}$ open, 
$f\in\mathcal{C}^{k}(\Omega,E)$ and $g\in\mathcal{C}^{k}(\Omega)$. 
Then $gf\in\mathcal{C}^{k}(\Omega,E)$ and 
\[
(\partial^{\beta})^{E}(gf)(x)
=\sum_{\gamma\leq \beta}{\dbinom{\beta}{\gamma}(\partial^{\beta-\gamma})^{\mathbb{K}}g(x)}(\partial^{\gamma})^{E}f(x),
\quad x\in \Omega,\; \beta\in\mathbb{N}_{0}^{d},\,|\beta|\leq k.
\]
\end{prop}

\begin{lem}\label{lem:comp.supp.dense}
Let $E$ be an lcHs, $k\in\mathbb{N}_{0,\infty}$ and $\mathcal{V}^{k}$ be a family of locally bounded weights 
on an open set $\Omega\subset\mathbb{R}^{d}$. 
If $\mathcal{CV}^{k}_{0}(\Omega,E)$ satisfies the cut-off criterion, then
the space $\mathcal{C}^{k}_{c}(\Omega,E)$ is dense in $\mathcal{CV}^{k}_{0}(\Omega,E)$.
\end{lem}
\begin{proof}
 The local boundedness of $\mathcal{V}^{k}$ yields that $\mathcal{C}^{k}_{c}(\Omega,E)$ 
 is a linear subspace of $\mathcal{CV}^{k}_{0}(\Omega,E)$ 
 by \prettyref{rem:compact.supp} which we equip with the induced topology.
 Let $f\in\mathcal{CV}^{k}_{0}(\Omega,E)$, $j\in J$, $l\in\langle k \rangle$ and $\alpha\in\mathfrak{A}$. 
 Due to the cut-off criterion there is $\delta>0$ such that for $\varepsilon>0$ 
 there is $K\subset\Omega$ compact with $(K+\overline{\mathbb{B}_{\delta}(0)}\,\bigr)\subset \Omega$ 
 and $|f|_{\Omega\setminus K,j,l,\alpha}<\varepsilon$. 
 We choose a cut-off function $\psi\in\mathcal{C}^{\infty}_{c}(\Omega)$ 
 with $0\leq\psi\leq 1$ so that $\psi=1$ in a neighbourhood of $K$ and 
 \[
 \bigl|(\partial^{\beta})^{\mathbb{K}}\psi\bigr|\leq C_{\beta}\delta^{-|\beta|}
 \]
 on $\Omega$ for all $\beta\in\mathbb{N}_{0}^{d}$ where $C_{\beta}>0$ only depends on $\beta$ 
 (see \cite[Theorem 1.4.1, p.\ 25]{H1}). 
 We set $K_{0}:=\operatorname{supp}\psi$, note that $\psi f\in\mathcal{C}^{k}_{c}(\Omega,E)$ by the product rule and 
 \begin{flalign*}
 &\quad |f-\psi f|_{j,l,\alpha}\\
 &=\sup_{\substack{x\in\Omega\setminus K\\ \beta\in\mathbb{N}_{0}^{d},|\beta|\leq l}}
 p_{\alpha}\bigl((\partial^{\beta})^{E}(f-\psi f)(x)\bigr)\nu_{j,l}(x)\\
  &\leq \sup_{\substack{x\in\Omega\setminus K\\ \beta\in\mathbb{N}_{0}^{d},|\beta|\leq l}}
  p_{\alpha}\bigl((\partial^{\beta})^{E}f(x)\bigr)\nu_{j,l}(x) 
  + \sup_{\substack{x\in\Omega\setminus K\\ \beta\in\mathbb{N}_{0}^{d},|\beta|\leq l}}
  p_{\alpha}\bigl((\partial^{\beta})^{E}(\psi f)(x)\bigr)\nu_{j,l}(x) \\
  &=|f|_{\Omega\setminus K,j,l,\alpha}+\sup_{\substack{x\in(\Omega\setminus K)\cap K_{0}\\ \beta\in\mathbb{N}_{0}^{d},|\beta|\leq l}}
  p_{\alpha}\bigl(\sum_{\gamma\leq \beta}{\dbinom{\beta}{\gamma}(\partial^{\beta-\gamma})^{\mathbb{K}}
   \psi(x)}(\partial^{\gamma})^{E}f(x)\bigr)\nu_{j,l}(x)\\
  &\leq |f|_{\Omega\setminus K,j,l,\alpha}
  +\sup_{\substack{z\in K_{0}\\ \beta\in\mathbb{N}_{0}^{d},|\beta|\leq l}}
  \sum_{\gamma\leq \beta}{\dbinom{\beta}{\gamma}\bigl|(\partial^{\beta-\gamma})^{\mathbb{K}}
  \psi(z)}\bigr|
  \Bigl(\sup_{\substack{x\in\Omega\setminus K\\ \tau\in\mathbb{N}_{0}^{d},|\tau|\leq l}}
   p_{\alpha}\bigl((\partial^{\tau})^{E}f(x)\bigr)\nu_{j,l}(x)\Bigr)\\
  &\leq |f|_{\Omega\setminus K,j,l,\alpha}+\underbrace{\sup_{\beta\in\mathbb{N}_{0}^{d},|\beta|\leq l}
  \sum_{\gamma\leq \beta}{\dbinom{\beta}{\gamma}C_{\beta-\gamma}\delta^{-|\beta-\gamma|}}}_{=:C_{l,\delta}<\infty}
   |f|_{\Omega\setminus K,j,l,\alpha}\\
  &=(1+C_{l,\delta})|f|_{\Omega\setminus K,j,l,\alpha}<(1+C_{l,\delta})\varepsilon.
 \end{flalign*}
 The independence of $C_{l,\delta}$ from $\varepsilon$ implies the statement.
\end{proof}

We complete this section by pointing out the link between our question on finite dimensional approximation 
and the tensor product. If $\mathcal{V}^{k}$ is locally bounded away from zero, 
there is a nice relation between our spaces of vector-valued functions 
and the $\varepsilon$-product which uses that the point-evaluation functionals $\delta_{x}\colon f\mapsto f(x)$ are continuous 
on $\mathcal{CV}^{k}(\Omega)$ by our definition of a weight.

\begin{prop}\label{prop:eps-prod.embed}
Let $E$ be an lcHs, 
$k\in\mathbb{N}_{0,\infty}$, $\mathcal{V}^{k}$ be a family of weights 
which is locally bounded away from zero on a locally compact Hausdorff space $\Omega$ ($k=0$) or 
an open set $\Omega\subset\mathbb{R}^{d}$ ($k>0$).
\begin{enumerate}
\item [a)] In addition, let $\mathcal{CV}^{k}_{0}(\Omega)$ be barrelled if $k>0$. Then
\[
 S_{\mathcal{CV}^{k}_{0}(\Omega)}\colon \mathcal{CV}^{k}_{0}(\Omega)\varepsilon E\to\mathcal{CV}^{k}_{0}(\Omega,E),\; 
 u\longmapsto [x\mapsto u(\delta_{x})],
\]
is an isomorphism into, i.e.\ an isomorphism to its range.
\item [b)] In addition, let $\mathcal{CV}^{k}(\Omega)$ be barrelled if $k>0$. Then
\[
 S_{\mathcal{CV}^{k}(\Omega)}\colon \mathcal{CV}^{k}(\Omega)\varepsilon E\to\mathcal{CV}^{k}(\Omega,E),\;
 u\longmapsto [x\mapsto u(\delta_{x})],
\]
is an isomorphism into.
\end{enumerate}
\end{prop}
\begin{proof}
As a simplification we omit the index of $S$.
Let $u\in \mathcal{CV}^{k}_{0}(\Omega)\varepsilon E$ resp.\ $\mathcal{CV}^{k}(\Omega)\varepsilon E$. 
The continuity of $S(u)$ is a consequence of \cite[4.1 Proposition, p.\ 18]{kruse2017} and 
\cite[4.2 Lemma (i), p.\ 19]{kruse2017} since $\mathcal{V}^{k}$ is locally bounded away from zero.
If $k>0$, then the continuous partial differentiability of $S(u)$ up to order $k$ follows
from \cite[4.12 Proposition, p.\ 22]{kruse2017}
as $\mathcal{CV}^{k}_{0}(\Omega)$ resp.\ $\mathcal{CV}^{k}(\Omega)$ is barrelled
and $\mathcal{V}^{k}$ locally bounded away from zero. 
If $u\in\mathcal{CV}^{k}_{0}(\Omega)\varepsilon E$, then $S(u)$ vanishes together with all its derivatives 
when weigthed at infinity by \cite[4.13 Proposition, p.\ 23]{kruse2017}. 
Thanks to these observations \cite[3.9 Theorem, p.\ 9]{kruse2017} proves our statement.
\end{proof}

In particular, if $J$ is countable and $\mathcal{V}^{k}$ locally bounded away from zero, 
then the Fr\'{e}chet spaces $\mathcal{CV}^{k}(\Omega)$ and $\mathcal{CV}^{k}_{0}(\Omega)$ are barrelled. 
This result allows us to identify the injective tensor product of 
$\mathcal{CV}^{k}(\Omega)$ resp.\ $\mathcal{CV}^{k}_{0}(\Omega)$ and $E$ with 
a subspace of $\mathcal{CV}^{k}(\Omega,E)$ resp.\ $\mathcal{CV}^{k}_{0}(\Omega,E)$.
Let us use the symbol $\mathcal{F}$ for $\mathcal{CV}^{k}$ or $\mathcal{CV}^{k}_{0}$. 
We consider $\mathcal{F}(\Omega)\otimes E$ as an algebraic subspace of $\mathcal{F}(\Omega)\varepsilon E$ 
by means of the linear injection 
\[
\Theta_{\mathcal{F}(\Omega)}\colon \mathcal{F}(\Omega)\otimes E\to \mathcal{F}(\Omega)\varepsilon E,\; 
\sum^{m}_{n=1}{f_{n}\otimes e_{n}}\longmapsto\bigl[y\mapsto \sum^{m}_{n=1}{y(f_{n}) e_{n}}\bigr].
\]
Via $\Theta_{\mathcal{F}(\Omega)}$ the topology of $\mathcal{F}(\Omega)\varepsilon E$ induces a locally convex topology on 
$\mathcal{F}(\Omega)\otimes E$ and $\mathcal{F}(\Omega)\otimes_{\varepsilon} E$ denotes 
$\mathcal{F}(\Omega)\otimes E$ equipped with this topology. 
From the preceding proposition and the composition $S_{\mathcal{F}(\Omega)}\circ\Theta_{\mathcal{F}(\Omega)}$ 
we obtain:

\begin{cor}\label{cor:tensor.embed}
Let $E$ be an lcHs, $k\in\mathbb{N}_{0,\infty}$, $\mathcal{V}^{k}$ be a family of weights 
which is locally bounded away from zero on a locally compact Hausdorff space $\Omega$ ($k=0$) or 
an open set $\Omega\subset\mathbb{R}^{d}$ ($k>0$). Fix the notation $\mathcal{F}=\mathcal{CV}^{k}$ or $\mathcal{CV}^{k}_{0}$ 
and let $\mathcal{F}(\Omega)$ be barrelled if $k>0$.
\begin{enumerate}
 \item [a)] We get by identification of isomorphic subspaces
 \[
  \mathcal{F}(\Omega)\otimes_{\varepsilon} E\subset \mathcal{F}(\Omega)\varepsilon E \subset \mathcal{F}(\Omega,E)
 \]
 and the embedding $\mathcal{F}(\Omega)\otimes E\hookrightarrow \mathcal{F}(\Omega,E)$ is given by 
 $f\otimes e\mapsto[x\mapsto f(x)e]$.
 \item [b)] Let $\mathcal{F}(\Omega)$ and $E$ be complete. 
 If $\mathcal{F}(\Omega)\otimes E$ is dense in $\mathcal{F}(\Omega,E)$, then 
 \[
 \mathcal{F}(\Omega,E)\cong \mathcal{F}(\Omega)\varepsilon E\cong \mathcal{F}(\Omega)\widehat{\otimes}_{\varepsilon} E.
 \]
 In particular, $\mathcal{F}(\Omega)$ has the approximation property if $\mathcal{F}(\Omega)\otimes E$ is dense 
 in $\mathcal{F}(\Omega,E)$ for every complete $E$.
\end{enumerate}
\end{cor}
\begin{proof}
\begin{enumerate}
 \item [a)] The inclusions obviously hold by \prettyref{prop:eps-prod.embed} and 
 $\mathcal{F}(\Omega)\varepsilon E$ and $\mathcal{F}(\Omega,E)$ 
 induce the same topology on $\mathcal{F}(\Omega)\otimes E$. 
 Further, we have
 \[
 f\otimes e \overset{\Theta_{\mathcal{F}(\Omega)}}{\longmapsto}[y\mapsto y(f)e]
 \overset{S_{\mathcal{F}(\Omega)}}{\longmapsto}[x\longmapsto [y\mapsto y(f) e](\delta_{x})]
 =[x\mapsto f(x)e].
 \]
 \item [b)] If $\mathcal{F}(\Omega)$ and $E$ are complete, 
 then we obtain that $\mathcal{F}(\Omega)\varepsilon E$ is complete 
 by \cite[Satz 10.3, p.\ 234]{Kaballo}. 
 In addition, we get the completion of $\mathcal{F}(\Omega)\otimes_{\varepsilon} E$ 
 as its closure in $\mathcal{F}(\Omega)\varepsilon E$ which coincides 
 with the closure in $\mathcal{F}(\Omega,E)$. The rest follows directly from a).
\end{enumerate}
\end{proof}

Looking at part a), we derive 
\[
(S_{\mathcal{F}(\Omega)}\circ\Theta_{\mathcal{F}(\Omega)})\bigl(\sum^{m}_{n=1}{f_{n}\otimes e_{n}}\bigr)
=\sum^{m}_{n=1}{f_{n}e_{n}}
\]
for $m\in\mathbb{N}$, $f_{n}\in\mathcal{F}(\Omega)$ and $e_{n}\in E$, $1\leq n\leq m$. 
Hence we see that the answer to our question is affirmative if $\mathcal{F}(\Omega)\otimes E$ 
is dense in $\mathcal{F}(\Omega,E)$. For the sake of completeness we remark the following.

\begin{prop}\label{prop:eps-prod.isom}
Let $E$ be an lcHs, 
$k\in\mathbb{N}_{0,\infty}$, $\mathcal{V}^{k}$ be a family of weights 
which is locally bounded away from zero on a locally compact Hausdorff space $\Omega$ ($k=0$) or 
an open set $\Omega\subset\mathbb{R}^{d}$ ($k>0$).
\begin{enumerate}
\item [a)] In addition, let $\mathcal{CV}^{k}_{0}(\Omega)$ be barrelled if $k>0$. If $E$ is quasi-complete 
and $\mathcal{V}^{k}$ locally bounded on $\Omega$, then
\[
 \mathcal{CV}^{k}_{0}(\Omega)\varepsilon E\cong \mathcal{CV}^{k}_{0}(\Omega,E)\quad\text{via}\;S_{\mathcal{CV}^{k}_{0}(\Omega)}.
\]
\item [b)] In addition, let $\mathcal{CV}^{k}(\Omega)$ be barrelled if $k>0$. If $E$ is a semi-Montel space, then
\[
 \mathcal{CV}^{k}(\Omega)\varepsilon E\cong\mathcal{CV}^{k}(\Omega,E)\quad\text{via}\;S_{\mathcal{CV}^{k}(\Omega)}.
\]
\end{enumerate}
\end{prop}
\begin{proof}
For $k>0$ this is \cite[5.10 Example a), p.\ 28]{kruse2017} resp.\ \cite[3.21 Example a), p.\ 14]{kruse2017}.  
Statement a) for $k=0$ is a consequence of \cite[3.20 Corollary, p.\ 13]{kruse2017} in combination with 
\cite[4.1 Proposition, p.\ 18]{kruse2017}, \cite[4.2 Lemma (i), p.\ 19]{kruse2017} and \cite[4.13 Proposition, p.\ 23]{kruse2017}. 
For $k=0$ statement b) follows from \cite[3.19 Corollary, p.\ 13]{kruse2017} in combination with 
\cite[4.1 Proposition, p.\ 18]{kruse2017} and \cite[4.2 Lemma (i), p.\ 19]{kruse2017}.
\end{proof}

The corresponding results for $k=0$ and a Nachbin-family $\mathcal{V}^{0}$ of weights are given in 
\cite[2.4 Theorem, p.\ 138-139]{B2} and \cite[2.12 Satz, p.\ 141]{B2}.
In combination with our preceding observation, we deduce that every element of $\mathcal{CV}^{k}_{0}(\Omega,E)$ 
can be approximated in $\mathcal{CV}^{k}_{0}(\Omega,E)$ by functions with values in a finite dimensional subspace 
if $E$ is a quasi-complete space with approximation property and the assumptions of the proposition above are fulfilled. 
The same is true for $\mathcal{CV}^{k}(\Omega,E)$ if $E$ is a semi-Montel space with approximation property. 
Due to the strong conditions on $E$ this is not really satisfying but actually the best we get for 
general $\mathcal{CV}^{k}(\Omega,E)$. For $\mathcal{CV}^{k}_{0}(\Omega,E)$ there is a better result available 
whose proof we prepare on the next pages.

\section{Convolution via the Pettis-integral}

In this section we review the notion of the Pettis-integral. Tr\`{e}ves uses the Riemann-integral 
to define the convolution $f\ast g$ of a function $f\in\mathcal{C}_{c}^{k}(\Omega,E)$ and a function 
$g\in \mathcal{C}^{\infty}_{c}(\mathbb{R}^{d})$ in the proof of Theorem \ref{thm:treves} 
and states (without a proof) that the convolution defined in this way is a function in
$\mathcal{C}_{c}^{\infty}(\mathbb{R}^{d},\widehat{E})$ and 
has all the properties known from the convolution of two scalar-valued functions. 
We use the Pettis-integral instead to define the convolution. The reason is that we can use the dominated convergence 
theorem for the Pettis-integral \cite[Theorem 2, p.\ 162-163]{musial1987} to get the Leibniz' rule 
for differentiation under the integral sign which enables us to prove that the convolution 
has some of the key properties known from the scalar-valued case.

Let us fix some notation first. 
For a measure space $(X, \Sigma, \mu)$ let 
\[
\mathfrak{L}^{1}(X,\mu):=\{f\colon X\to\mathbb{K}\;\text{measurable}\;|\;
q_{1}(f):=\int_{X}|f(x)|\d\mu(x)<\infty\}
\]
and define the quotient space of integrable functions with respect to measure $\mu$ by 
$\mathcal{L}^{1}(X,\mu):=\mathfrak{L}^{1}(X,\mu)/\{f\in\mathfrak{L}^{1}(X,\mu)\;|\;q_{1}(f)=0\}$.
From now on we do not distinguish between equivalence classes and 
their representants anymore.
We say that $f\colon X\to \mathbb{K}$ is integrable on $\Lambda\in\Sigma$ 
and write $f\in\mathcal{L}^{1}(\Lambda,\mu)$ if $\chi_{\Lambda}f\in \mathcal{L}^{1}(X,\mu)$ 
where $\chi_{\Lambda}$ is the characteristic function of $\Lambda$. Then we set 
\[
\int\limits_{\Lambda} f(x) \d\mu(x):=\int\limits_{X} \chi_{\Lambda}(x)f(x)  \d\mu(x).
\] 

\begin{defn}[Pettis-integral]
\label{def:integral}
	Let $(X, \Sigma, \mu)$ be a measure space and $E$ an lcHs. 
	A function $f: X \to E$ is called weakly (scalarly) measurable if the function 
	$e'\circ f: X \to \mathbb{K}$, $(e'\circ f)(x) := \langle e' , f(x) \rangle:=e'(f(x))$, 
	is measurable for all $e' \in E'$.  
	A weakly measurable function is said to be weakly (scalarly) integrable 
	if $ e' \circ f \in \mathcal{L}^{1}(X, \mu)$. 
	A function $f\colon X\to E$ is called Pettis-integrable on $\Lambda\in\Sigma$ 
 	if it is weakly integrable on $\Lambda$ and
	\[
	\exists\; e_{\Lambda} \in E \; \forall\; e' \in E':\; \langle e' , e_{\Lambda} \rangle 
	= \int\limits_{\Lambda} \langle e' , f(x) \rangle \d\mu(x). 
	\]  
	In this case $e_{\Lambda}$ is unique due to $E$ being Hausdorff and we set 
	\[
	 \int\limits_{\Lambda} f(x) \d\mu(x):=e_{\Lambda}.
	\]	
	A function $f$ is called Pettis-integrable on $\Sigma$ if it is Pettis-integrable on all $\Lambda\in\Sigma$.
\end{defn}	 

We write $\mathcal{N}_{\mu}$ for the set of $\mu$-null sets of a measure space $(X, \Sigma, \mu)$ 
and for $\Lambda\in\Sigma$ we use the notion 
$(\Lambda,\Sigma_{\mid\Lambda},\mu_{\mid\Lambda})$ for the restricted measure space given by 
$\Sigma_{\mid\Lambda}:=\{\omega\in\Sigma\;|\;\omega\subset\Lambda\}$ and $\mu_{\mid\Lambda}:=\mu_{\mid\Sigma_{\mid\Lambda}}$. 
If we consider the measure space $(\mathbb{R}^{d}, \mathscr{L}(\mathbb{R}^{d}), \lambda)$ of Lebesgue measurable sets, 
we just write $\d x:=\d\lambda(x)$.

\begin{rem}\label{rem:int.auf.traeger}
Let $(X, \Sigma, \mu)$ be a measure space, $E$ an lcHs and $f$ Pettis-integrable 
on $\Lambda\in \Sigma$. If $\omega\in\Sigma$ such that $\omega\subset\Lambda$ and 
$(\Lambda\setminus\omega)\subset\{x\in X\;|\;f(x)=0\}$, then $f$ is Pettis-integrable on $\omega$ 
and 
\begin{equation}\label{rem:int.auf.traeger.1}
\int\limits_{\omega} f(x) \d\mu(x)=\int\limits_{\Lambda} f(x) \d\mu(x).
\end{equation}
This follows directly from 
\[
 \langle e' ,  \int\limits_{\Lambda} f(x) \d\mu(x)\rangle
 =\int\limits_{\Lambda} \langle e' , f(x) \rangle\d\mu(x)
 =\int\limits_{\omega} \langle e' , f(x) \rangle\d\mu(x),\quad e'\in E'.
\]
\end{rem}

\begin{lem}\label{lem:cont.para}
Let $E$ be a quasi-complete lcHs, $(X, \Sigma, \mu)$ a measure space, $T$ a metric space
and suppose that $f\colon X\times T\to E$ fulfils the following conditions.
\begin{enumerate}
 \item [a)] $f(\cdot,t)$ is Pettis-integrable on $\Sigma$ for all $t\in T$,
 \item [b)] $f(x,\cdot)\colon T\to E$ is continuous in a point $t_{0}\in T$ for $\mu$-almost all $x\in X$,
 \item [c)] there is a neighbourhood $U\subset T$ of $t_{0}$ and a Pettis-integrable function $\psi$ on $\Sigma$ 
 such that 
 \[
  \forall\;t\in U,\, e'\in E'\;\exists\;N\in\mathcal{N}_{\mu}\;\forall\; x\in X\setminus N:\; 
  |\langle e' , f(x,t) \rangle\bigr|\leq |\langle e' , \psi(x) \rangle|.
 \]
\end{enumerate}
Then $g_{\Lambda}\colon T\to E$, $g_{\Lambda}(t):=\int\limits_{\Lambda} f(x,t) \d\mu(x)$, is well-defined 
and continuous in $t_{0}$ for every $\Lambda\in\Sigma$. 
\end{lem}
\begin{proof}
Let $\Lambda\in\Sigma$ and $(t_{n})$ be a sequence in $U$ converging to $t_{0}$. 
From the continuous dependency of a scalar integral on a parameter (see \cite[5.6 Satz, p.\ 147]{elstrodt2005}) 
we derive
\begin{equation}\label{lem:cont.para.1}
 \lim_{n\to\infty}\int_{\Lambda}\langle e' ,\underbrace{f(x,t_{n})}_{=:f_{n}(x)}\rangle \d\mu(x)
 = \int_{\Lambda}\langle e' ,\underbrace{f(x,t_{0})}_{=:\widetilde{f}(x)}\rangle \d\mu(x).
\end{equation}
For $n\in\mathbb{N}$ and $e'\in E'$ there is $N\in\mathcal{N}_{\mu}$ such that
\begin{equation}\label{lem:cont.para.2}
|\langle e' ,f_{n}(x) \rangle|=|\langle e' ,f(x,t_{n}) \rangle|\leq |\langle e' , \psi(x) \rangle|
\end{equation}
for every $x\in X\setminus N$. 
Due to \eqref{lem:cont.para.1} for every $\Lambda\in\Sigma$ and $e'\in E'$,  
\eqref{lem:cont.para.2} and the quasi-completeness of $E$
we can apply the dominated convergence theorem for the Pettis-integral \cite[Theorem 2, p.\ 162-163]{musial1987} 
and deduce
\[
 \lim_{n\to\infty}g_{\Lambda}(t_{n})
 =\lim_{n\to\infty}\int\limits_{\Lambda}{f}_{n}(x) \d\mu(x)
 =\int\limits_{\Lambda} \widetilde{f}(x) \d\mu(x)
 =g_{\Lambda}(t_{0}).
\]
\end{proof}

The next lemma is the Leibniz' rule for differentiation under the integral sign for the Pettis-integral.

\begin{lem}[{Leibniz' rule}]\label{lem:Leibniz}
Let $E$ be a quasi-complete lcHs, $(X, \Sigma, \mu)$ a measure space, $T\subset\mathbb{R}^{d}$ open 
and suppose that $f\colon X\times T\to E$ fulfils the following conditions.
\begin{enumerate}
 \item [a)] $f(\cdot,t)$ is Pettis-integrable on $\Sigma$ for all $t\in T$,
 \item [b)] there is a $\mu$-null set $N_{0}\in\mathcal{N}_{\mu}$ with $f(x,\cdot)\in\mathcal{C}^{1}(T,E)$ 
 for all $x\in X\setminus N_{0}$,
 \item [c)] for every $j\in\mathbb{N}$, $1\leq j\leq d$, there is a Pettis-integrable function $\psi_{j}$ on $\Sigma$ such that 
 \[
 \qquad\quad \forall\;e'\in E'\;\exists\; N\in\mathcal{N}_{\mu}\;\forall\;x\in X\setminus (N\cup N_{0}):\;
 \bigl|(\partial_{t_{j}})^{\mathbb{K}}\langle e' , f(x,\cdot) \rangle\bigr|\leq |\langle e' , \psi_{j}(x) \rangle|.
 \]
\end{enumerate}
Then $g_{\Lambda}\colon T\to E$, $g_{\Lambda}(t):=\int\limits_{\Lambda} f(x,t) \d\mu(x)$, is well-defined 
for every $\Lambda\in\Sigma$, $g_{\Lambda}\in \mathcal{C}^{1}(T,E)$ and%, $(\partial_{t_{j}})^{E}f(\cdot,t)$ is 
%Pettis-integrable on $\Sigma$ for all $t\in T$ 
\[
(\partial_{t_{j}})^{E}g_{\Lambda}(t)=\int\limits_{\Lambda} (\partial_{t_{j}})^{E}f(x,t) \d\mu(x),\quad t\in T.
\]
\end{lem}
\begin{proof}
First, we consider the case $\mathbb{K}=\mathbb{R}$. Let $\Lambda\in\Sigma$, $j\in\mathbb{N}$, $1\leq j\leq d$, $t\in T$ 
and $(h_{n})$ be a real sequence converging to $0$ such that $h_{n}\neq 0$ and $t+h_{n}e_{j}\in T$ for all $n$ where 
$e_{j}$ is the $j$-th unit vector in $\mathbb{R}^{d}$. Then 
\[
\frac{g_{\Lambda}(t+h_{n}e_{j})-g_{\Lambda}(t)}{h_{n}}
=\int_{\Lambda}\underbrace{\frac{f(x,t+h_{n}e_{j})-f(x,t)}{h_{n}}}_{=:f_{n}(x)} \d\mu(x).
\] 
We define the function $\widetilde{f}\colon X\to E$ given by $\widetilde{f}(x):=(\partial_{t_{j}})^{E}f(x,t)$ 
for $x\in X\setminus N_{0}$ and $\widetilde{f}(x):=0$ 
for $x\in N_{0}$. We observe that 
\begin{align}\label{thm:Leibniz.1}
\lim_{n\to\infty}\int_{\Lambda}\langle e' ,f_{n}(x) \rangle \d\mu(x)
&=\int_{\Lambda}(\partial_{t_{j}})^{\mathbb{K}}\langle e' ,f(x,t) \rangle \d\mu(x)\notag\\
&=\int_{\Lambda}\langle e' ,\partial_{t_{j}}^{E}f(x,t) \rangle \d\mu(x)
=\int_{\Lambda}\langle e' ,\widetilde{f}(x) \rangle \d\mu(x)
\end{align}
holds for every $e'\in E'$ where we used the scalar Leibniz' rule for differentiation under the integral sign for the first equation 
which can be applied due to our assumptions (see  \cite[5.7 Satz, p.\ 147-148]{elstrodt2005}). 
For $e'\in E'$ there is $N\in \mathcal{N}_{\mu}$ such that for every 
$x\in X\setminus (N\cup N_{0})$ and $n\in\mathbb{N}$ there is $\theta\in[0,1]$ with
\[
\langle e' ,f_{n}(x) \rangle=\frac{\langle e' ,f(x,t+h_{n}e_{j})\rangle-\langle e' ,f(x,t)\rangle}{h_{n}}
=(\partial_{t_{j}})^{\mathbb{K}}\langle e' ,f(x,t+\theta h_{n}e_{j})\rangle
\]
by the mean value theorem ($\mathbb{K}=\mathbb{R}$) implying 
\begin{equation}\label{thm:Leibniz.2}
|\langle e' ,f_{n}(x) \rangle|=|(\partial_{t_{j}})^{\mathbb{K}}\langle e' ,f(x,t+\theta h_{n}e_{j})\rangle|
\leq |\langle e' , \psi_{j}(x) \rangle|.
\end{equation}
Due to \eqref{thm:Leibniz.1} for every $\Lambda\in\Sigma$ and $e'\in E'$, 
\eqref{thm:Leibniz.2} and the quasi-completeness of $E$ we can 
apply the dominated convergence theorem for the Pettis-integral \cite[Theorem 2, p.\ 162-163]{musial1987} again
and obtain that $\widetilde{f}$ is Pettis-integrable on $\Sigma$ plus 
\begin{align*}
(\partial_{t_{j}})^{E}g_{\Lambda}(t)&=\lim_{n\to\infty}\frac{g_{\Lambda}(t+h_{n}e_{j})-g_{\Lambda}(t)}{h_{n}}
=\lim_{n\to\infty}\int_{\Lambda}f_{n}(x) \d\mu(x)=\int_{\Lambda}\widetilde{f}(x)\d\mu(x)\\
&=\int\limits_{\Lambda} (\partial_{t_{j}})^{E}f(x,t) \d\mu(x).
\end{align*}
The continuity of $(\partial_{t_{j}})^{E}g_{\Lambda}$ follows from \prettyref{lem:cont.para} 
by replacing $f$ with $(\partial_{t_{j}})^{E}f$.
For $\mathbb{K}=\mathbb{C}$ we just have to substitute $\langle e' ,\cdot\rangle$ by $\operatorname{Re}\langle e' ,\cdot\rangle$ 
(real part) and $\operatorname{Im}\langle e' ,\cdot\rangle$ (imaginary part) in the considerations above.
\end{proof}

Now, we are able to define the convolution of a vector-valued and a scalar-valued continuous function via the 
Pettis-integral, if one of them has compact support, and to show some of its basic properties which are known from 
the convolution of scalar-valued functions (scalar convolution). 
For the properties of the scalar convolution see e.g.\ \cite[Chap.\ 26, p.\ 278-283]{Treves}.

\begin{lem}\label{lem:convolution}
Let $E$ be a quasi-complete lcHs, $k,n\in\mathbb{N}_{0,\infty}$, $f\in\mathcal{C}^{k}(\mathbb{R}^{d},E)$ 
and $g\in\mathcal{C}^{n}(\mathbb{R}^{d})$, either one having compact support. 
The convolution
\[
 f\ast g\colon\mathbb{R}^{d}\to E,\;(f\ast g)(x):=\int_{\mathbb{R}^{d}}f(y)g(x-y)\d y,
\]
is well-defined, $\operatorname{supp} (f\ast g)\subset \operatorname{supp}f+\operatorname{supp}g$, 
$f\ast g=g\ast f$, where 
\[
 g\ast f\colon\mathbb{R}^{d}\to E,\;(g\ast f)(x):=\int_{\mathbb{R}^{d}}g(y)f(x-y)\d y,
\]
and $f\ast g\in \mathcal{C}^{n}(\mathbb{R}^{d},E)$ plus
\begin{align}
 (\partial^{\beta})^{E}(f\ast g)&=f\ast \bigl((\partial^{\beta})^{\mathbb{K}}g\bigr) \quad &&,|\beta|\leq n,\label{lem:convolution.1}\\
 (\partial^{\beta})^{E}(f\ast g)&=\bigl((\partial^{\beta})^{E}f\bigr)\ast g &&,|\beta|\leq \min(k,n).\label{lem:convolution.2}
\end{align}
\end{lem}
\begin{proof}
Let $h\colon \mathbb{R}^{d}\times\mathbb{R}^{d}\to E$, $h(y,x):=f(y)g(x-y)$. First, 
we show that $h(\cdot,x)$ is Pettis-integrable on $\mathscr{L}(\mathbb{R}^{d})$ for every $x\in\mathbb{R}^{d}$ 
implying that $f\ast g$ is well-defined. 
We note that $\langle e' ,h(\cdot,x) \rangle\in \mathcal{L}^{1}(\mathbb{R}^{d},\lambda)$ for every $e'\in E'$ 
and $x\in \mathbb{R}^{d}$. 
Let $x\in\mathbb{R}^{d}$ and $\Lambda\in\mathscr{L}(\mathbb{R}^{d})$. We define the linear map
\[
I_{\Lambda,x}\colon E'\to \mathbb{K},\;I_{\Lambda,x}(e'):=\int_{\Lambda}\langle e' ,h(y,x) \rangle\d y.
\]
Setting $K_{f}:=\operatorname{supp}f$ and $K_{g}:=\operatorname{supp}g$, we observe that 
\[
I_{\Lambda,x}(e')=\int_{\Lambda\cap K_{f}}\langle e' ,f(y)g(x-y) \rangle \d y
=\int_{\Lambda\cap(x-K_{g})}\langle e' ,f(y)g(x-y) \rangle \d y.
\]
If $K_{f}=\operatorname{supp}f$ is compact, we get 
\[
|I_{\Lambda,x}(e')|\leq \lambda(K_{f})\sup\bigl\{|e'(z)|\;|\;z\in f(K_{f})g(x-K_{f})\bigr\}.
\]
The set $f(K_{f})g(x-K_{f})$ is compact in $E$ and thus the closure of its absolutely 
convex hull is compact in $E$ as well by \cite[9-2-10 Example, p.\ 134]{Wilansky} 
because $E$ is quasi-complete. Hence it follows that $I_{\Lambda,x}\in(E_{\kappa}')'\cong E$ by the theorem 
of Mackey-Arens meaning that there is $e_{\Lambda}(x)\in E$ such that 
\[
\langle e',e_{\Lambda}(x) \rangle=I_{\Lambda,x}(e')=\int_{\Lambda}\langle e' ,h(y,x) \rangle\d y
\]
for all $e'\in E'$. Thus $h(\cdot,x)$ is Pettis-integrable on 
$\mathscr{L}(\mathbb{R}^{d})$ and 
\[
(f\ast g)(x)=e_{\mathbb{R}^{d}}(x)\underset{\eqref{rem:int.auf.traeger.1}}{=}e_{K_{f}}(x)=e_{x-K_{g}}(x)
\]
for every $x\in\mathbb{R}^{d}$ if $K_{f}=\operatorname{supp}f$ is compact. 
If $K_{g}=\operatorname{supp}g$ is compact, then 
the estimate 
\[
|I_{\Lambda,x}(e')|\leq \lambda(x-K_{g})\sup\bigl\{|e'(z)|\;|\;z\in f(x-K_{g})g(K_{g})\bigr\}
\]
yields to the Pettis-integrability in the same manner.

Let $x\notin\operatorname{supp}f+\operatorname{supp}g$. If $y\notin\operatorname{supp}f$, then $h(y,x)=0$. 
If $y\in\operatorname{supp}f$, then $x-y\notin \operatorname{supp}g$ and thus $h(y,x)=0$. 
Hence we have $h(\cdot,x)=0$ implying 
$\operatorname{supp} (f\ast g)\subset \operatorname{supp}f+\operatorname{supp}g$. 
%Analogously to $h$ we see that $\widetilde{h}(\cdot,x)$ is Pettis-integrable on $\mathscr{L}(\mathbb{R}^{d})$ 
%for every $x\in\mathbb{R}^{d}$ where $\widetilde{h}\colon \mathbb{R}^{d}\times\mathbb{R}^{d}\to E$, 
%$\widetilde{h}(y,x):=f(x-y)g(y)$. Hence $g\ast f$ is well-defined and 
From
\begin{align*}
\langle e' ,(f\ast g)(x)\rangle
&=\int_{\mathbb{R}^{d}}\langle e' ,f(y)g(x-y)\rangle\d y
=\int_{\mathbb{R}^{d}}\langle e' ,f(y)\rangle g(x-y)\d y\\
&=\bigl((e'\circ f)\ast g\bigr)(x)
=\bigl(g\ast(e'\circ f)\bigr)(x)
=\int_{\mathbb{R}^{d}}\langle e' ,g(y)f(x-y)\rangle\d y
\end{align*}
for every $x\in\mathbb{R}^{d}$ and $e'\in E'$, where 
we used the commutativity of scalar convolution for the fourth equation, 
follows 
\[
 (f\ast g)(x)=e_{\mathbb{R}^{d}}(x)=(g\ast f)(x)
\]
for every $x\in\mathbb{R}^{d}$.

Next, we show that $f\ast g\in\mathcal{C}^{n}(\mathbb{R}^{d},E)$ and \eqref{lem:convolution.1} holds 
by applying \prettyref{lem:cont.para} and \ref{lem:Leibniz}. So we have to check that the conditions a)-c) 
of these lemmas are fulfilled.
First, fix $x_{0}\in\mathbb{R}^{d}$, let $\varepsilon>0$ and $\beta\in\mathbb{N}_{0}^{d}$, $|\beta|\leq n$. 
If $K_{f}=\operatorname{supp}f$ is compact, we set 
$h_{f,\beta}:=(\partial_{x}^{\beta})^{E}h_{\mid K_{f}\times \mathbb{B}_{\varepsilon}(x_{0})}$ 
and observe that $h_{\mid K_{f}\times \mathbb{B}_{\varepsilon}(x_{0})}(y,\cdot)
\in\mathcal{C}^{n}(\mathbb{B}_{\varepsilon}(x_{0}),E)$ for every $y\in K_{f}$ (condition b)). 
It follows from the theorem of Mackey-Arens and
\[
\bigl|\int_{\omega}\langle e' ,h_{f,\beta}(y,x)\rangle\d y\bigr|
\leq\lambda(K_{f})\sup\bigl\{|e'(z)|\;|\;z\in f(K_{f})(\partial^{\beta})^{\mathbb{K}}
g(\overline{\mathbb{B}_{\varepsilon}(x_{0})}-K_{f})\bigr\}
\]
for every $e'\in E'$, $\omega\in\mathscr{L}(\mathbb{R}^{d})_{\mid K_{f}}$ and $x\in \mathbb{B}_{\varepsilon}(x_{0})$ 
that $h_{f,\beta}(\cdot,x)$ is Pettis-integrable on $\mathscr{L}(\mathbb{R}^{d})_{\mid K_{f}}$ 
for every $x\in\mathbb{B}_{\varepsilon}(x_{0})$ (condition a)).
Now, we check that condition c) is satisfied. We observe that the estimate 
\[
\bigl|\int_{\omega}\langle e' ,f(y)\rangle\d y\bigr|\leq\lambda(K_{f})\sup\bigl\{|e'(z)|\;|\;z\in f(K_{f})\bigr\}
\]
for every $e'\in E'$ and $\omega\in\mathscr{L}(\mathbb{R}^{d})_{\mid K_{f}}$ implies 
that $f_{\mid K_{f}}$ is Pettis-integrable on $\mathscr{L}(\mathbb{R}^{d})_{\mid K_{f}}$ due to the theorem 
of Mackey-Arens again.
The inequality 
\begin{align*}
|\langle e' ,h_{f,\beta}(y,x)\rangle|&=|\langle e' ,f(y)(\partial_{x}^{\beta})^{\mathbb{K}}[x\mapsto g(x-y)]\rangle|\\
&\leq |\langle e' ,f(y)\rangle|\sup\bigl\{|(\partial^{\beta})^{\mathbb{K}}g(z)|\;|\;
z\in \overline{\mathbb{B}_{\varepsilon}(x_{0})}-K_{f}\bigr\}\\
&\leq |\langle e' ,q_{\overline{\mathbb{B}_{\varepsilon}(x_{0})}-K_{f},n}(g)\cdot f(y)\rangle|
\end{align*}
for every $e'\in E'$ and $(y,x)\in K_{f}\times \mathbb{B}_{\varepsilon}(x_{0})$ with the seminorm 
$q_{\overline{\mathbb{B}_{\varepsilon}(x_{0})}-K_{f},n}$ from \eqref{cv-complete.0} 
yields to condition c) being satisfied. Hence $f\ast g\in\mathcal{C}^{n}(\mathbb{B}_{\varepsilon}(x_{0}),E)$ by 
\prettyref{lem:cont.para} if $n=0$ and by \prettyref{lem:Leibniz} if $n=1$ as well as 
\begin{align*}
\partial_{x_{j}}^{E}(f\ast g)(x)
&=\partial_{x_{j}}^{E}[x\mapsto\int_{\mathbb{R}^{d}}f(y)g(x-y)\d y]\underset{\eqref{rem:int.auf.traeger.1}}{=} 
\partial_{x_{j}}^{E}[x\mapsto\int_{K_{f}}f(y)g(x-y)\d y]\\
&=\int_{K_{f}}f(y)(\partial_{x_{j}})^{\mathbb{K}}[x\mapsto g(x-y)]\d y
\underset{\eqref{rem:int.auf.traeger.1}}{=} \int_{\mathbb{R}^{d}}f(y)(\partial^{e_{j}})^{\mathbb{K}}g(x-y)\d y\\
&=\bigl(f\ast \bigl((\partial^{e_{j}})^{\mathbb{K}}g\bigr)\bigr)(x)
\end{align*}
for every $x\in\mathbb{B}_{\varepsilon}(x_{0})$. Letting $\varepsilon\to\infty$, we obtain the 
result for $n=0$ and $n=1$ if $K_{f}=\operatorname{supp}f$ is compact. 
For $n\geq 2$ it follows from induction over the order $|\beta|$. 
If $K_{g}=\operatorname{supp}g$ is compact, the same approach with
$h_{g,\beta}:=(\partial_{x}^{\beta})^{E}h_{\mid K_{g}\times \mathbb{B}_{\varepsilon}(x_{0})}$ 
instead of $h_{f,\beta}$ proves the statement. 
Furthermore, for $|\beta|\leq\min(k,n)$ we get
\begin{flalign*}
&\quad\langle e',(\partial^{\beta})^{E}(f\ast g)(x)\rangle\\
&=\int_{\mathbb{R}^{d}}\langle e',f(y)(\partial^{\beta})^{\mathbb{K}}g(x-y)\rangle\d y
=\int_{\mathbb{R}^{d}}(e'\circ f)(y)(\partial^{\beta})^{\mathbb{K}}g(x-y)\d y\\
&=\bigl((e'\circ f)\ast\bigl((\partial^{\beta})^{\mathbb{K}}g\bigr)\bigr)(x)
=\bigl((\partial^{\beta})^{\mathbb{K}}(e'\circ f)\ast g\bigr)(x)\\
&=\bigl((e'\circ(\partial^{\beta})^{E} f)\ast g\bigr)(x)
=\int_{\mathbb{R}^{d}}\langle e',(\partial^{\beta})^{E}f(y)g(x-y)\rangle\d y
\end{flalign*}
for every $e'\in E'$ and $x\in\mathbb{R}^{d}$, where we used the corresponding result for 
the scalar convolution for the fourth equation, implying 
$(\partial^{\beta})^{E}(f\ast g)=\bigl((\partial^{\beta})^{E}f\bigr)\ast g$.
\end{proof}

Looking at the lemma above, we see that it differs a bit from the properties known from 
the convolution of two scalar-valued functions. 
It is an open problem whether we actually have $f\ast g\in \mathcal{C}^{\max(k,n)}(\mathbb{R}^{d},E)$ 
and \eqref{lem:convolution.2} for $|\beta|\leq k$ 
under the assumptions of the lemma. But since we only apply the lemma above in the case $n=\infty$, this does not affect us.

We recall the construction of a molifier from \cite[p.\ 155-156]{Treves}. Let 
\[
 \rho\colon\mathbb{R}^{d}\to \mathbb{R},\; 
 \rho(x):=\begin{cases} C \exp(-\tfrac{1}{1-|x|^2})&,|x|<1,\\
           0&,|x|\geq 1,
          \end{cases}
\]
where $C:=\bigl(\int_{\mathbb{B}_{1}(0)}\exp(-\tfrac{1}{1-|x|^2})\d x\bigr)^{-1}$. For $n\in\mathbb{N}$ 
we define the molifier $\rho_{n}$ given by 
$\rho_{n}(x):=n^{d}\rho(nx)$, $x\in\mathbb{R}^{d}$. Then we have $\rho_{n}\in\mathcal{C}^{\infty}_{c}(\mathbb{R}^{d})$, 
$\rho_{n}\geq 0$, $\operatorname{supp}\rho_{n}=\overline{\mathbb{B}_{1/n}(0)}$ 
and $\int_{\mathbb{R}^{d}}\rho_{n}(x)\d x=1$.

We can extend a function $f\in\mathcal{C}^{k}_{c}(\Omega,E)$, $k\in\mathbb{N}_{0,\infty}$ 
and $\Omega\subset\mathbb{R}^{d}$, to a function 
$f_{\operatorname{ex}}\in\mathcal{C}^{k}_{c}(\mathbb{R}^{d},E)$ by setting 
$f_{{\operatorname{ex}}}:=f$ on $\Omega$ and $f_{\operatorname{ex}}:=0$ on $\mathbb{R}^{d}\setminus\Omega$. 
In this way the convolution $f\ast g:=(f_{\operatorname{ex}}\ast g)_{\mid\Omega}$ with a 
function $g\in\mathcal{C}(\mathbb{R}^{d})$ is a well-defined 
function on $\Omega$ if $E$ is quasi-complete, and
we have the following approximation by regularisation in analogy  
to the scalar-valued case (see e.g.\ \cite[Chap.\ 15, Corollary 1, p.\ 158]{Treves}).

\begin{lem}\label{lem:regularization}
Let $E$ be a quasi-complete lcHs, $k\in\mathbb{N}_{0,\infty}$, $\mathcal{V}^{k}$ be a family of locally bounded weights 
on an open set $\Omega\subset\mathbb{R}^{d}$ and $f\in\mathcal{C}^{k}_{c}(\Omega,E)$. 
Then $(f\ast\rho_{n})$ converges to $f$ 
in $\mathcal{CV}^{k}_{0}(\Omega,E)$ as $n\to \infty$.
\end{lem}
\begin{proof}
Due to \prettyref{lem:convolution} we obtain that 
$f_{\operatorname{ex}}\ast\rho_{n}\in\mathcal{C}^{\infty}_{c}(\mathbb{R}^{d},E)$ 
for every $n\in\mathbb{N}$. Since $\mathcal{V}^{k}$ is locally bounded on $\Omega$, 
we derive $f\ast\rho_{n}\in \mathcal{CV}^{k}_{0}(\Omega,E)$. 
Let $\varepsilon>0$, $j\in J$, $l\in\langle k \rangle$ and $\alpha\in\mathfrak{A}$. 
For $\beta\in\mathbb{N}_{0}^{d}$, $|\beta|\leq l$, 
there is $\delta_{\beta}>0$ such that for all $x\in\Omega$ and $y\in\mathbb{R}^{d}$ with 
$|y|=|(x-y)-x|\leq \delta_{\beta}$ we have
\begin{equation}\label{lem:regularization.1}
p_{\alpha}\bigl((\partial^{\beta})^{E}f_{\operatorname{ex}}(x-y)-(\partial^{\beta})^{E}f(x)\bigr)<\varepsilon
\end{equation}
because the function $(\partial^{\beta})^{E}f_{\operatorname{ex}}$ is uniformly continuous on whole $\mathbb{R}^{d}$ 
as it is continuous with compact support. Therefore we deduce for all $n>1/\delta_{\beta}$ that 
$\operatorname{supp}\rho_{n}=\overline{\mathbb{B}_{1/n}(0)}
\subset \overline{\mathbb{B}_{\delta_{\beta}}(0)}$ and hence
\begin{flalign*}
&\quad\;\; p_{\alpha}\bigl((\partial^{\beta})^{E}(f\ast\rho_{n}-f)(x)\bigr)\\
&\underset{\eqref{lem:convolution.2}}{=}p_{\alpha}\bigl(\bigl((\partial^{\beta})^{E}f\bigr)\ast\rho_{n}(x)
-(\partial^{\beta})^{E}f(x)\bigr)\\
&\;\;=p_{\alpha}\bigl(\rho_{n}\ast\bigl((\partial^{\beta})^{E}f\bigr)(x)
-(\partial^{\beta})^{E}f(x)\bigr)\\
&\;\;=p_{\alpha}\bigr(\int_{\mathbb{R}^{d}}(\partial^{\beta})^{E}
f_{\operatorname{ex}}(x-y)\rho_{n}(y)\d y-(\partial^{\beta})^{E}f(x)\bigr) \\
&\;\;=p_{\alpha}\bigl(\int_{\mathbb{R}^{d}}(\partial^{\beta})^{E}f_{\operatorname{ex}}(x-y)\rho_{n}(y)
-(\partial^{\beta})^{E}f(x)\rho_{n}(y)\d y \bigr)\\
&\;\underset{\eqref{rem:int.auf.traeger.1}}{=}p_{\alpha}\bigl(\int_{\overline{\mathbb{B}_{1/n}(0)}}(\partial^{\beta})^{E}f_{\operatorname{ex}}(x-y)\rho_{n}(y)
-(\partial^{\beta})^{E}f(x)\rho_{n}(y)\d y \bigr)\\
&\underset{\eqref{lem:regularization.1}}{\leq}\varepsilon\int_{\mathbb{R}^{d}}\rho_{n}(y)\d y=\varepsilon
\end{flalign*}
by \prettyref{lem:convolution} for every $x\in\Omega$. As $0\in\operatorname{supp}\rho_{n}$, we get that 
\[
\operatorname{supp}(\partial^{\beta})^{E}(f\ast\rho_{n}-f)\subset( \operatorname{supp}f+\operatorname{supp}\rho_{n})
=\bigl( \operatorname{supp}f+\overline{\mathbb{B}_{1/n}(0)}\bigr) 
\]
for every $|\beta|\leq l$ and $n\in\mathbb{N}$ by virtue of \prettyref{lem:convolution}. 
Since $\operatorname{supp}f\subset\Omega$ is compact and $\Omega$ open, there is $r>0$ such that 
$\bigl( \operatorname{supp}f+\overline{\mathbb{B}_{r}(0)}\bigr)\subset\Omega$ yielding
\[
\operatorname{supp}(\partial^{\beta})^{E}(f\ast\rho_{n}-f)
\subset \bigl( \operatorname{supp}f+\overline{\mathbb{B}_{r}(0)}\bigr)=:K
\]
for all $n\geq 1/r$.
Choosing $\delta:=\min\{\delta_{\beta}\;|\;\beta\in\mathbb{N}_{0}^{d},\,|\beta|\leq l\}>0$, 
we obtain for all $n>\max\{1/\delta,1/r\}$ that
\begin{align*}
\bigl|f\ast\rho_{n}-f\bigr|_{j,l,\alpha}
&=\sup_{\substack{x\in K\\ \beta\in\mathbb{N}_{0}^{d},|\beta|\leq l}}
p_{\alpha}\bigl((\partial^{\beta})^{E}(f\ast\rho_{n}-f)(x)\bigr)\nu_{j,l}(x)
\leq\varepsilon \sup_{x\in K}\nu_{j,l}(x)
\end{align*}
which implies our statement since $\mathcal{V}^{k}$ is locally bounded on $\Omega$ and $K\subset\Omega$ 
is compact.
\end{proof}
 
\section{Approximation property}

Finally, we dedicate our last section to our main theorem. We start with the case $k=0$.

\begin{prop}\label{prop:localization}
Let $E$ be an lcHs and $\mathcal{V}^{0}$ a family of locally bounded weights 
which is locally bounded away from zero on a locally compact Hausdorff space $\Omega$. 
Then the following statements hold.
\begin{enumerate}
\item [a)] $\mathcal{C}^{0}_{c}(\Omega)\otimes E$ is dense in $\mathcal{CV}^{0}_{0}(\Omega,E)$.
\item [b)] For any $f\in\mathcal{C}^{0}_{c}(\Omega,E)$ and any open neighbourhood $V$ of $\operatorname{supp}f$, 
for every $\varepsilon>0$, $j\in J$ and $\alpha\in\mathfrak{A}$, 
there is $g\in\mathcal{C}^{0}_{c}(\Omega)\otimes E$ such 
that $\operatorname{supp}g\subset V$ and $|f-g|_{j,0,\alpha}\leq\varepsilon$.
\item [c)] If $E$ is complete, then 
\[
 \mathcal{CV}^{0}_{0}(\Omega,E)\cong\mathcal{CV}^{0}_{0}(\Omega)\varepsilon E
 \cong\mathcal{CV}^{0}_{0}(\Omega)\widehat{\otimes}_{\varepsilon} E.
\]
\item [d)] $\mathcal{CV}^{0}_{0}(\Omega)$ has the approximation property.
\end{enumerate}
\end{prop}
\begin{proof}
First, we consider part a). Due to \prettyref{cor:tensor.embed} a) and \prettyref{rem:compact.supp}
$\mathcal{C}^{0}_{c}(\Omega)\otimes E$ can be identified with a subspace of $\mathcal{CV}^{0}_{0}(\Omega,E)$ 
equipped with the induced topology since $\mathcal{V}^{0}$ is locally bounded and locally bounded away from zero.\\ 
Let $f\in\mathcal{CV}^{0}_{0}(\Omega,E)$, $\varepsilon>0$, $j\in J$ and 
$\alpha\in\mathfrak{A}$ and we fix the notation $\nu_{j}:=\nu_{j,0}$. 
Then there is a compact set $\widetilde{K}\subset\Omega$ such that 
\[
|f|_{\Omega\setminus \widetilde{K},j,0,\alpha}=
\sup_{x\in \Omega\setminus \widetilde{K}}p_{\alpha}(f(x))
 \nu_{j}(x)<\varepsilon .
\]
Let $K:=\widetilde{K}$. Since $\Omega$ is locally compact, every $w\in K$ has an open, 
relatively compact neighbourhood $U_{w}\subset\Omega$. 
As $K$ is compact and $K\subset \bigcup_{w\in K} U_{w}$, there are $m\in\mathbb{N}$ and $w_{i}\in K$, $1\leq i\leq m$, such that
\[
 K\subset\bigcup_{i=1}^{m} U_{w_{i}}=:W\subset\Omega.
\]
The set $W$ is open and relatively compact because it is a finite union of open, relatively compact sets. 
The local boundedness of $\mathcal{V}^{0}$ and relative compactness of $W$ imply that 
\[
N:=1+\sup_{x\in \overline{W}}\nu_{j}(x)<\infty.
\]
For $x\in K$ we define $V_{x}:=\{y\in\Omega\;|\; p_{\alpha}(f(y)-f(x))<\frac{\varepsilon}{N}\}$. 
Thus we have $V_{x}=f^{-1}\left(B_{\alpha}(f(x),\frac{\varepsilon}{N})\right)$, 
where $B_{\alpha}(f(x),\frac{\varepsilon}{N}):=\{e\in E\;|\;p_{\alpha}(e-f(x))<\frac{\varepsilon}{N}\}$, 
implying that $V_{x}$ is open in $\Omega$ since $f$ is continuous. 
Hence we get $K\subset \bigcup_{x\in K}V_{x}$ and conclude that there are $n\in\mathbb{N}$ and 
$x_{i}\in K$, $1\leq i\leq n$, such that $K\subset \bigcup_{i=1}^{n}V_{x_{i}}$ from the compactness of $K$. 
We note that
\begin{equation}\label{prop:localization.1}
 K=(K\cap \overline{W})\subset\bigcup_{i=1}^{n}(V_{x_{i}}\cap \overline{W}) .
\end{equation}
The sets $V_{x_{i}}\cap \overline{W}$ are open in the compact Hausdorff space 
$\overline{W}$ with respect to the topology induced by $\Omega$. 
Since the compact Hausdorff space $\overline{W}$ is normal by \cite[Chap.\ IX, \S4.1, Proposition 1, p.\ 181]{bourbakiII} 
and $K$ is closed in $\overline{W}$, there is a family of non-negative real-valued continuous functions 
$\left(\varphi_{i}\right)$ with $\operatorname{supp}\varphi_{i}\subset (V_{x_{i}}\cap \overline{W})$ such that 
$\sum_{i=1}^{n}\varphi_{i}=1$ on $K$ and $\sum_{i=1}^{n}\varphi_{i}\leq 1$ on $\overline{W}$ 
by \cite[Chap.\ IX, \S4.3, Corollary, p.\ 186]{bourbakiII}. 
By trivially extending $\varphi_{i}$ on $\Omega\setminus\overline{W}$, we obtain $\varphi_{i}\in\mathcal{C}^{0}_{c}(\Omega)$ 
because $\overline{W}$ is compact. 
We define 
\[
g:=\sum_{i=1}^{n}\varphi_{i}\otimes f(x_{i})\in\mathcal{C}^{0}_{c}(\Omega)\otimes E 
\]
and observe $\operatorname{supp}g \subset \bigcup_{i=1}^{n} (V_{x_{i}}\cap \overline{W})$.
If $x\in K$, then $\varphi_{i}(x)p_{\alpha}(f(x)-f(x_{i}))=0$ if $x\notin V_{x_{i}}\cap \overline{W}$, and 
\begin{align*}
p_{\alpha}(f(x)-g(x))&=
p_{\alpha}\bigl(\sum_{i=1}^{n}\varphi_{i}(x)(f(x)-f(x_{i}))\bigr)
\leq \sum_{i=1}^{n}\varphi_{i}(x)p_{\alpha}(f(x)-f(x_{i}))\\
&\leq \sum_{i=1}^{n}\varphi_{i}(x)\frac{\varepsilon}{N}
=\frac{\varepsilon}{N}
\end{align*}
yielding to
\begin{align*}
\sup_{x\in K}p_{\alpha}((f-g)(x))\nu_{j}(x)
 &\leq \sup_{x\in K}\frac{\varepsilon}{N}\nu_{j}(x)\leq \sup_{x\in \overline{W}}\frac{\varepsilon}{N}\nu_{j}(x)
 =\frac{\varepsilon}{N}\cdot(N-1)<\varepsilon .
\end{align*}
If $x\notin K$, then $\varphi_{i}(x)f(x_{i})=0$ if $x\notin (V_{x_{i}}\cap \overline{W})\setminus K$. 
If $x\in (V_{x_{i}}\cap \overline{W})\setminus K$, then
\begin{align*}
p_{\alpha}(\varphi_{i}(x)f(x_{i}))
&\leq\varphi_{i}(x)\bigl(p_{\alpha}(f(x_{i})-f(x))+p_{\alpha}(f(x))\bigr)
\leq \varphi_{i}(x)\bigl(\frac{\varepsilon}{N}+p_{\alpha}(f(x))\bigr)
\end{align*}
yielding to
\begin{flalign*}
&\;\;\;\:|f-g|_{\Omega\setminus K,j,0,\alpha}\\
&=\sup_{x\in \Omega\setminus K}p_{\alpha}((f-g)(x))
 \nu_{j}(x)
 \leq\sup_{x\in \Omega\setminus K}
 \bigl(p_{\alpha}(f(x))+p_{\alpha}(g(x))\bigr)\nu_{j}(x)\\
 &\leq \varepsilon+
 \sup_{x\in \Omega\setminus K}\sum_{i=1}^{n}p_{\alpha}(\varphi_{i}(x)f(x_{i}))\nu_{j}(x)
 \leq \varepsilon+
 \sup_{x\in \Omega\setminus K}\sum_{i=1}^{n}\varphi_{i}(x)\bigl(\frac{\varepsilon}{N}+p_{\alpha}(f(x))\bigr)\nu_{j}(x)\\
 &\leq 2\varepsilon+\frac{\varepsilon}{N}\sup_{x\in \Omega\setminus K}\sum_{i=1}^{n}\varphi_{i}(x)\nu_{j}(x)
 \leq 2\varepsilon+\frac{\varepsilon}{N}\sup_{x\in \overline{W}}\sum_{i=1}^{n}\varphi_{i}(x)\nu_{j}(x)
 \leq 2\varepsilon+\frac{\varepsilon}{N}\cdot (N-1)<3\varepsilon
\end{flalign*}
implying 
\[
\left|f-g\right|_{j,0,\alpha}< 4\varepsilon
\]
which proves part a).

Part c) follows from a) and \prettyref{cor:tensor.embed} b) because $\mathcal{CV}^{0}_{0}(\Omega)$ 
is complete by \prettyref{prop:cv-complete}. Part d) is implied by part c). 
Let us turn to part b). Let $f\in\mathcal{C}^{0}_{c}(\Omega,E)$ and $V$ be an open neighbourhood 
of $\widetilde{K}:=\operatorname{supp}f$. Then we can replace \eqref{prop:localization.1} by 
\[
K=(K\cap V\cap\overline{W})\subset\bigcup_{i=1}^{n}(V_{x_{i}}\cap V\cap \overline{W}) 
\]
and then the open sets $V_{x_{i}}$ by the open sets $V_{x_{i}}\cap V$ in what follows \eqref{prop:localization.1}. 
This gives 
\[
\operatorname{supp}g \subset \bigl(\bigcup_{i=1}^{n} (V_{x_{i}}\cap V\cap \overline{W})\bigr)\subset V
\]
proving b).
\end{proof}

If $\Omega$ is an open subset of $\mathbb{R}^{d}$, we can choose a smooth partition of unity 
(see e.g.\ \cite[Theorem 1.4.5, p.\ 28]{H1}) and even get that $\mathcal{C}^{\infty}_{c}(\Omega)\otimes E$ is dense 
in $\mathcal{CV}^{0}_{0}(\Omega,E)$ under the assumptions of the proposition above.

The proof of part a) is a modification of the proof of \cite[5.1 Satz, p.\ 204]{B1} by Bierstedt. 
Since $\Omega$ is locally compact and not just a completely regular Hausdorff space, 
we can use the partition of unity from \cite[Chap.\ IX, \S4.1, Proposition 1, p.\ 181]{bourbakiII}. 
Bierstedt has to use the partition of unity from \cite[23, Lemma 2, p.\ 71]{nachbin1965} 
and due to the assumptions of this lemma he can not choose $K=\widetilde{K}$ but has to use
\[
  K':=\{x\in\Omega\;|\;p_{\alpha}(f(x))\nu_{j}(x)\geq\varepsilon\}\subset \widetilde{K} .
\]
Bierstedt's assumption that $\nu_{j}$ is upper semi-continuous guarantees that $K'$ is closed 
and thus compact as a closed subset of the compact set $\widetilde{K}$.  
Choosing $K:= K'$, the proof above works as well where the existence 
of the open set $W\subset\Omega$ is a consequence of the upper semi-continuity of $\nu_{j}$ again. 
Comparing \prettyref{thm:bierstedt} and \prettyref{prop:localization}, 
we see that \prettyref{thm:bierstedt} is far more general concerning the 
spaces $\Omega$ involved but the condition of $\mathcal{V}^{0}$ being a locally bounded family 
in \prettyref{prop:localization} is weaker than the condition 
of being a family of upper semi-continuous weights in \prettyref{thm:bierstedt}. 
Let us phrase our main theorem.

\begin{thm}\label{thm:weighted.diff}
Let $E$ be an lcHs, 
$k\in\mathbb{N}_{\infty}$ and $\mathcal{V}^{k}$ be a family of locally bounded weights 
which is locally bounded away from zero on an open set $\Omega\subset\mathbb{R}^{d}$. 
Let $\mathcal{CV}^{k}_{0}(\Omega)$ be barrelled and $\mathcal{C}^{k}_{c}(\Omega,E)$ 
dense in $\mathcal{CV}^{k}_{0}(\Omega,E)$.
Then the following statements hold.
\begin{enumerate}
\item [a)] $\mathcal{C}^{\infty}_{c}(\Omega)\otimes E$ is dense in $\mathcal{CV}^{k}_{0}(\Omega,E)$.
\item [b)] If $E$ is complete, then 
\[
 \mathcal{CV}^{k}_{0}(\Omega,E)\cong\mathcal{CV}^{k}_{0}(\Omega)\varepsilon E
 \cong\mathcal{CV}^{k}_{0}(\Omega)\widehat{\otimes}_{\varepsilon} E.
\]
\item [c)] $\mathcal{CV}^{k}_{0}(\Omega)$ has the approximation property.
\end{enumerate}
\end{thm}
\begin{proof}
It suffices to prove part a) because 
part b) follows from a) and \prettyref{cor:tensor.embed} b)
since $\mathcal{CV}^{k}_{0}(\Omega)$ is complete by \prettyref{prop:cv-complete}. 
Then part c) is a consequence of b). 
Let us turn to part a). Since $\mathcal{CV}^{k}_{0}(\Omega)$ is barrelled, 
$\mathcal{V}^{k}$ locally bounded and locally bounded away from zero, 
the space $\mathcal{C}_{c}^{\infty}(\Omega)\otimes E$ 
can be considered as a topological subspace of $\mathcal{CV}_{0}^{k}(\Omega)\otimes_{\varepsilon}E$ 
by \prettyref{cor:tensor.embed} a) and \prettyref{rem:compact.supp} when equipped with the induced topology.

Let $f\in\mathcal{CV}^{k}_{0}(\Omega,E)$, $\varepsilon>0$, $j\in J$, $l\in\langle k \rangle$ and 
$\alpha\in\widehat{\mathfrak{A}}$ where $(p_{\alpha})_{\alpha\in\widehat{\mathfrak{A}}}$ is the system of
seminorms describing the locally convex topology of the completion $\widehat{E}$ of $E$. 
In the following we consider functions with values in $E$ also as functions with values in $\widehat{E}$ 
and note that $\mathcal{CV}_{0}^{k}(\Omega,\widehat{E})$ is the completion of 
$\mathcal{CV}_{0}^{k}(\Omega,E)$ by \prettyref{prop:cv-complete}. 
Thus the topologies of $\mathcal{CV}^{k}_{0}(\Omega,E)$ and 
$\mathcal{CV}^{k}_{0}(\Omega,\widehat{E})$ coincide on $\mathcal{CV}^{k}_{0}(\Omega,E)$. 
The density of $\mathcal{C}^{k}_{c}(\Omega,E)$ in $\mathcal{CV}_{0}^{k}(\Omega,E)$ yields that 
there is $\widetilde{f}\in\mathcal{C}^{k}_{c}(\Omega,E)$ such that $|f-\widetilde{f}|_{j,l,\alpha}<\varepsilon/3$. 
Further, there is $N_{0}\in\mathbb{N}$ with
$|\widetilde{f}-\widetilde{f}\ast\rho_{n}|_{j,l,\alpha}<\varepsilon/3$ 
for all $n\geq N_{0}$ by \prettyref{lem:regularization} as $\widehat{E}$ is complete. Let $K_{1}:=\operatorname{supp}\widetilde{f}$ 
and choose an open neighbourhood $V$ of $K_{1}$ such that $V$ is relatively compact in $\Omega$ 
which is possible since $K_{1}$ is compact and $\Omega\subset\mathbb{R}^{d}$ open. Since 
$\mathcal{V}^{k}$ is locally bounded away from zero, there is $i\in J$ such that 
\[
C_{1}:=\sup_{x\in \overline{V}}\nu_{i,0}(x)^{-1}=\bigl(\inf_{x\in \overline{V}}\nu_{i,0}(x)\bigr)^{-1}<\infty.
\]
From the relative compactness of $V$ in $\Omega$ follows that there is $N_{1}\in\mathbb{N}$ such that 
\[
\overline{V}+\overline{\mathbb{B}_{1/n}(0)}\subset \Omega
\]
for all $n\geq N_{1}$. Choosing $N_{2}:=\max\{N_{0},N_{1}\}$ and defining the compact set 
$K_{2}:=\overline{V}+\overline{\mathbb{B}_{1/N_{2}}(0)}\subset\Omega$, we get that 
\[
C_{2}:=\sup_{x\in K_{2}}\nu_{j,l}(x)<\infty
\]
because $\mathcal{V}^{k}$ is locally bounded. Further, we estimate 
\[
C_{3}:=\sup_{\beta\in\mathbb{N}_{0}^{d},|\beta|\leq l}\int_{\mathbb{R}^{d}}\bigl|\partial^{\beta}\rho_{N_{2}}(y)\bigr|\d y
\leq (N_{2})^{l}\sup_{\beta\in\mathbb{N}_{0}^{d},|\beta|\leq l}\int_{\mathbb{R}^{d}}\bigl|\partial^{\beta}\rho(y)\bigr|\d y
<\infty.
\]
By virtue of \prettyref{prop:localization} b) there is 
$g=\sum_{m=1}^{q}g_{m}\otimes e_{m}\in\mathcal{C}^{0}_{c}(\Omega)\otimes E$ such that 
$\operatorname{supp}g\subset V$ and 
\[
|\widetilde{f}-g|_{i,0,\alpha}<\frac{\varepsilon}{3C_{1}C_{2}C_{3}}.
\]
By \prettyref{lem:convolution} we observe that $g\ast\rho_{N_{2}} \in\mathcal{C}^{\infty}_{c}(\Omega, E)$ 
with 
\[
\operatorname{supp}(g\ast \rho_{N_{2}})\subset\overline{V}+\overline{\mathbb{B}_{1/N_{2}}(0)}
=K_{2}\subset\Omega
\]
and 
\[
 g\ast\rho_{N_{2}}=\sum_{m=1}^{q} (g_{m}\ast\rho_{N_{2}})\otimes e_{m}\in\mathcal{C}^{\infty}_{c}(\Omega)\otimes E.
\]
Thus we have by \prettyref{lem:convolution}
\[
\operatorname{supp}(\widetilde{f}\ast \rho_{N_{2}})\subset \overline{V}+\overline{\mathbb{B}_{1/N_{2}}(0)}=K_{2}
\]
yielding
\[
\operatorname{supp}(\widetilde{f}\ast\rho_{N_{2}}-g\ast \rho_{N_{2}})\subset
\overline{V}+\overline{\mathbb{B}_{1/N_{2}}(0)}=K_{2}\subset\Omega
\]
and 
\begin{flalign*}
&\quad\;|\widetilde{f}\ast\rho_{N_{2}}-g\ast\rho_{N_{2}}|_{j,l,\alpha}\\
&\;\;\mathclap{\underset{\eqref{lem:convolution.1}}{=}}\;\;\sup_{\substack{x\in K_{2}\\\beta\in\mathbb{N}_{0}^{d},|\beta|\leq l}} 
p_{\alpha}\bigl((\widetilde{f}-g)\ast(\partial^{\beta}\rho_{N_{2}})(x)\bigr)\nu_{j,l}(x)\\
&=\sup_{\substack{x\in K_{2}\\\beta\in\mathbb{N}_{0}^{d},|\beta|\leq l}} 
p_{\alpha}\bigl(\int_{\mathbb{R}^{d}}(\partial^{\beta}\rho_{N_{2}})(x-y)
\bigl(\widetilde{f}_{\operatorname{ex}}(y)-g_{\operatorname{ex}}(y)\bigr)\d y\bigr)\nu_{j,l}(x)\\
&\leq \sup_{\substack{x\in K_{2}\\\beta\in\mathbb{N}_{0}^{d},|\beta|\leq l}}
\int_{\mathbb{R}^{d}}\bigl|(\partial^{\beta}\rho_{N_{2}})(x-y)\bigr|\d y
\sup_{\substack{z\in\operatorname{supp}(\widetilde{f})\\\;\;\cup\operatorname{supp}(g)}}
p_{\alpha}(\widetilde{f}(z)-g(z))\nu_{j,l}(x)\\
&= \sup_{\substack{x\in K_{2}\\\beta\in\mathbb{N}_{0}^{d},|\beta|\leq l}}
\int_{\mathbb{R}^{d}}\bigl|(\partial^{\beta}\rho_{N_{2}})(y)\bigr|\d y\,
\sup_{z\in\overline{V}}p_{\alpha}(\widetilde{f}(z)-g(z))\nu_{j,l}(x)\\
&\leq C_{3}\bigl(\sup_{x\in K_{2}}\nu_{j,l}(x)\bigr)
\bigl(\sup_{z\in\overline{V}}p_{\alpha}(\widetilde{f}(z)-g(z))\bigr)\\
&= C_{3}C_{2}\sup_{z\in\overline{V}}p_{\alpha}(\widetilde{f}(z)-g(z))\nu_{i,0}(z)\nu_{i,0}(z)^{-1}\\
&\leq C_{3}C_{2}C_{1}|\widetilde{f}-g|_{i,0,\alpha}<\frac{\varepsilon}{3}.
\end{flalign*}
Therefore we deduce
\[
|f-g\ast\rho_{N_{2}}|_{j,l,\alpha}
\leq|f-\widetilde{f}|_{j,l,\alpha}
+|\widetilde{f}-\widetilde{f}\ast\rho_{N_{2}}|_{j,l,\alpha}
+|\widetilde{f}\ast\rho_{N_{2}}-g\ast\rho_{N_{2}}|_{j,l,\alpha}
< \frac{\varepsilon}{3}+\frac{\varepsilon}{3}+\frac{\varepsilon}{3}=\varepsilon.
\]
Keeping in mind that $f\in\mathcal{CV}^{k}_{0}(\Omega,E)$ and 
$g\ast\rho_{N_{2}}\in\mathcal{C}^{\infty}_{c}(\Omega)\otimes E$, it follows
that $\mathcal{C}^{\infty}_{c}(\Omega)\otimes E$ is dense in $\mathcal{CV}^{k}_{0}(\Omega,E)$ with respect to the topology of 
$\mathcal{CV}^{k}_{0}(\Omega,\widehat{E})$. However, the latter space is just the completion of $\mathcal{CV}^{k}_{0}(\Omega,E)$ 
and thus the topologies of $\mathcal{CV}^{k}_{0}(\Omega,E)$ and 
$\mathcal{CV}^{k}_{0}(\Omega,\widehat{E})$ coincide on $\mathcal{CV}^{k}_{0}(\Omega,E)$. 
Hence we get that $\mathcal{C}^{\infty}_{c}(\Omega)\otimes E$ is dense in $\mathcal{CV}^{k}_{0}(\Omega,E)$. 
\end{proof}

$\mathcal{C}^{k}_{c}(\Omega,E)$ is dense in $\mathcal{CV}^{k}_{0}(\Omega,E)$ by \prettyref{lem:comp.supp.dense} 
if the latter space fufils the cut-off criterion and the family $\mathcal{V}^{k}$ is locally bounded. 
$\mathcal{CV}^{k}_{0}(\Omega)$ is a Fr\'{e}chet space and thus barrelled by \prettyref{prop:cv-complete} if the $J$ in 
$\mathcal{V}^{k}=(\nu_{j,l})_{j\in J,l\in\langle k \rangle}$ is countable. 
Let us complement what we said about the standard structure of a family of weights 
(see the remarks below \prettyref{def:weighted_diff_spaces})
by our additional conditions on the weights collected so far. 
The standard structure of a (countable) locally bounded family $\mathcal{V}^{k}$ which 
is bounded away from zero on a locally compact Hausdorff space $\Omega$ resp.\ 
on an open set $\Omega\subset\mathbb{R}^{d}$ is given by the following. 
Let $J:=\mathbb{N}$, $(\Omega_{j})_{j\in J}$ be a family of sets such that 
$\Omega_{j}\subset\Omega_{j+1}$ for all $j\in J$ with $\Omega=\bigcup_{j\in J}\Omega_{j}$ and 
\[
\forall\; K\subset\Omega\;\text{compact}\;\exists\;j\in J:\;K\subset\Omega_{j}.
\]
Let $\widetilde{\nu}_{j,l}\colon\Omega\to (0,\infty)$ be continuous for all $j\in J$, $l\in\langle k \rangle$
and increasing in $j\in J$ and in $l\in\langle k \rangle$ such that 
\begin{equation}\label{eq:standard.weight}
\nu_{j,l}(x)=\chi_{\Omega_{j}}(x)\widetilde{\nu}_{j,l}(x),\quad x\in \Omega,
\end{equation}
for every $j\in J$ and $l\in\langle k \rangle$ where $\chi_{\Omega_{j}}$ is the indicator function of $\Omega_{j}$. 
If $\Omega\neq \mathbb{R}^{d}$, then the cut-off criterion may add some restrictions 
on the structure of the sequence $(\Omega_{j})$, e.g.\ a positive distance from the boundary 
$\partial\Omega_{j}$ of $\Omega_{j}$ to the boundary of $\partial\Omega_{j+1}$ of $\Omega_{j+1}$ for all $j$.

\begin{exa}%\label{ex:weighted.spaces}
Let $E$ be an lcHs, $k\in\mathbb{N}_{\infty}$ and $\Omega\subset\mathbb{R}^{d}$ open. 
\prettyref{thm:weighted.diff} can be applied to the following spaces:
\begin{enumerate}
\item [a)] $\mathcal{C}^{k}(\Omega,E)$ with the topology of uniform convergence of 
all partial derivatives up to order $k$ on compact subsets of $\Omega$,
\item [b)] the Schwartz space $\mathcal{S}(\mathbb{R}^{d},E)$,
\item [c)] the space $\mathcal{O}_{M}(\mathbb{R}^{d},E)$ of multipliers of $\mathcal{S}(\mathbb{R}^{d})$,
\item [d)] let $\Omega_{j}:=\{x=(x_{1},x_{2})\in\mathbb{R}^{2}\;|\;1/(j+1)<|x_{2}|<j+1\}$ for all $j\in\mathbb{N}$ and
\[
 \qquad \qquad\mathcal{C}^{k}_{\operatorname{exp}}(\mathbb{R}^{2}\setminus\mathbb{R},E):=
 \bigl\{f\in\mathcal{C}^{k}(\mathbb{R}^{2}\setminus\mathbb{R},E)\;|\; 
 \forall\;j\in\mathbb{N},\,l\in\langle k \rangle,\, \alpha\in\mathfrak{A}:\;|f|_{j,l,\alpha}<\infty\bigr\} 
\]
 where
\[
 |f|_{j,l,\alpha}:=\sup_{\substack{(x_{1},x_{2})\in\Omega_{j}\\ \beta\in\mathbb{N}_{0}^{2},|\beta|\leq l}}
 p_{\alpha}\bigl((\partial^{\beta})^{E}f(x_{1},x_{2})\bigr)e^{-\frac{1}{j+1}|x_{1}|}.
 \]
\end{enumerate}
\end{exa}
\begin{proof}
\begin{enumerate}
\item [a)] From \prettyref{ex:standard_diff_spaces} a) we obtain $\mathcal{C}^{k}(\Omega,E)=\mathcal{CW}^{k}_{0}(\Omega,E)$ with
$\mathcal{W}^{k}:=\{\nu_{j,l}:=\chi_{\Omega_{j}}\;|\;j\in\mathbb{N},\,l\in\langle k \rangle\}$ 
where $(\Omega_{j})_{j\in\mathbb{N}}$ is a compact exhaustion of $\Omega$. 
The family of weights $\mathcal{W}^{k}$ is locally bounded and locally bounded away from zero. 
The Fr\'{e}chet space $\mathcal{C}^{k}(\Omega)$ is barrelled and the cut-off criterion is fulfilled 
by \prettyref{ex:cut-off}.
\item [b)] Due to \prettyref{ex:standard_diff_spaces} b) we have 
$\mathcal{S}(\mathbb{R}^{d},E)=\mathcal{CV}^{\infty}_{0}(\mathbb{R}^{d},E)$ with 
$\mathcal{V}^{\infty}:=\{\nu_{j,l}\;|\;j\in\mathbb{N},\,l\in\mathbb{N}_{0}\}$ where
$\nu_{j,l}(x):=(1+|x|^{2})^{l/2}$ for $x\in\mathbb{R}^{d}$. 
The family of weights is locally bounded and bounded away from zero, 
the Fr\'{e}chet space $\mathcal{S}(\mathbb{R}^{d})$ is barrelled and $\mathcal{S}(\mathbb{R}^{d},E)$ 
fulfils the cut-off criterion by \prettyref{ex:cut-off}.
\item [c)] The space of multipliers is defined by 
\[
\qquad\quad \mathcal{O}_{M}(\mathbb{R}^{d},E)
:=\{f\in \mathcal{C}^{\infty}(\mathbb{R}^{d},E)\;|\;\forall\;g\in\mathcal{S}(\mathbb{R}^{d}),\,
l\in\mathbb{N}_{0},\,\alpha\in \mathfrak{A}:\;\|f\|_{g,l,\alpha}<\infty\}
\]
where
\[
 \|f\|_{g,l,\alpha}:=\sup_{\substack{x\in\mathbb{R}^{d}\\
  \beta\in\mathbb{N}_{0}^{d},|\beta|\leq l}}
 p_{\alpha}\bigl((\partial^{\beta})^{E}f(x)\bigr)|g(x)|
\]
(see \cite[$3^{0}$), p.\ 97]{Schwartz1955}).
The space $\mathcal{O}_{M}(\mathbb{R}^{d})$ is barrelled by 
\cite[Chap.\ II, \S4, n$^\circ$4, Th\'{e}or\`{e}me 16, p.\ 131]{Gro}. 
Let $J:=\{j\subset\mathcal{S}(\mathbb{R}^{d})\;|\;j\;\text{finite}\}$ and 
define the family $\mathcal{V}^{\infty}$ of weights given by 
$\nu_{j,l}(x):=\max_{g\in j}|g(x)|$, $x\in\mathbb{R}^{d}$, for $j\in J$ and $l\in\mathbb{N}_{0}$. 
It is easily seen that the system of seminorms generated by 
\[
|f|_{j,l,\alpha}:=\sup_{\substack{x\in\mathbb{R}^{d}\\ \beta\in\mathbb{N}_{0}^{d}, |\beta|\leq l}}
 p_{\alpha}\bigl((\partial^{\beta})^{E}f(x)\bigr)\nu_{j,l}(x),\quad f\in \mathcal{O}_{M}(\mathbb{R}^{d},E),
\]
for $j\in J$, $l\in\mathbb{N}_{0}$ and $\alpha\in\mathfrak{A}$ induces the same topology on $\mathcal{O}_{M}(\mathbb{R}^{d},E)$. 
However, the family $\mathcal{V}^{\infty}$ is directed, locally bounded and bounded away from zero. 
Further, for every $\varepsilon>0$ there is $r>0$ such that $(1+|x|^{2})^{-1}<\varepsilon$ for all 
$x\notin\overline{\mathbb{B}_{r}(0)}=:K$ which implies for $j\in J$ and $l\in\mathbb{N}_{0}$ that 
\[
\nu_{j,l}(x)\leq\varepsilon \max_{g\in j}|g(x)(1+|x|^{2})|=\varepsilon\nu_{i,l}(x),\quad x\notin K,
\]
where $i:=\{g\cdot(1+|\cdot|^{2})\;|\;g\in j\}$ is a finite subset of $\mathcal{S}(\mathbb{R}^{d})$.
From \prettyref{rem:cv=cv_0} we conclude that $\mathcal{O}_{M}(\mathbb{R}^{d},E)
=\mathcal{CV}^{\infty}(\mathbb{R}^{d},E)=\mathcal{CV}^{\infty}_{0}(\mathbb{R}^{d},E)$.  
Due to \prettyref{rem:whole.cut-off} we note that $\mathcal{O}_{M}(\mathbb{R}^{d},E)$ satisifies the cut-off criterion.
\item [d)] The family $\mathcal{V}^{k}$ given by $\nu_{j,l}(x_{1},x_{2}):=\chi_{\Omega_{j}}(x_{1},x_{2})e^{-|x_{1}|/(j+1)}$, 
$(x_{1},x_{2})\in\mathbb{R}^{2}\setminus\mathbb{R}$, for $j\in\mathbb{N}$ and $l\in\langle k \rangle$ 
is locally bounded and bounded away from zero. For $j\in\mathbb{N}$ and $l\in\mathbb{N}_{0}$ we set $i:=2j+1$, $m:=l$, 
$\delta:=1/(2j+2)$ and for $0<\varepsilon<1$ we choose 
$K:=\{x=(x_{1},x_{2})\in\overline{\Omega_{j}}\;|\; |x_{1}|\leq -(\ln\varepsilon)(2j+2)\}$. 
This yields $\mathcal{C}^{k}_{\operatorname{exp}}(\mathbb{R}^{2}\setminus\mathbb{R},E)
=\mathcal{CV}^{k}(\mathbb{R}^{2}\setminus\mathbb{R},E)=\mathcal{CV}^{k}_{0}(\mathbb{R}^{2}\setminus\mathbb{R},E)$ 
by \prettyref{rem:cv=cv_0} and that the cut-off criterion is fulfilled. 
In addition, the Fr\'{e}chet space $\mathcal{C}^{k}_{\operatorname{exp}}(\mathbb{R}^{2}\setminus\mathbb{R})$ is barrelled.
\end{enumerate}
\end{proof}

Together with \prettyref{prop:localization} we get from example a) one of our starting points, 
namely \prettyref{thm:treves}, back. Example b) and c) are covered by 
\cite[Proposition 9, p.\ 108]{Schwartz1955} and \cite[Th\'{e}or\`{e}me 1, p.\ 111]{Schwartz1955}. 
The results b) and c) for the Schwartz space in example b) can also be found in 
\cite[Chap.\ II, \S3, n$^\circ$3, Exemples, p.\ 80-81]{Gro} 
with a different proof using the nuclearity of $\mathcal{S}(\mathbb{R}^{d})$. We complete this paper 
with a comparison of our conditions in \prettyref{thm:weighted.diff} with the ones stated 
by Schwartz in \cite{Schwartz1955} to get the same result for the spaces in example
a)-c) but only for $\Omega=\mathbb{R}^{d}$.

\begin{rem}
Schwartz treats the case $k>0$ and $\Omega=\mathbb{R}^{d}$ in \cite{Schwartz1955}. 
He assumes similar conditions $H_{1}$-$H_{4}$ for the space $\mathcal{H}^{k}(\mathbb{R}^{d}):=
\mathcal{H}^{k}(\mathbb{R}^{d},\mathbb{K})$ as we do (see \cite[p.\ 97-98]{Schwartz1955}). 
In $H_{1}$ the members of his family of weights $\Gamma$ 
are continuous and for every compact set $K\subset\mathbb{R}^{d}$ there is a weight in $\Gamma$ 
which is non-zero on $K$. $\mathcal{H}^{k}(\mathbb{R}^{d})$ is the space of functions $f\in\mathcal{C}^{k}(\mathbb{R}^{d})$ such that 
$\gamma \partial^{\beta}f$ is bounded on $\mathbb{R}^{d}$ for every $\gamma\in\Gamma$ and $|\beta|\leq k$. 
This yields to $\mathcal{C}^{k}_{c}(\mathbb{R}^{d})\subset\mathcal{H}^{k}(\mathbb{R}^{d})\subset
\mathcal{C}^{k}(\mathbb{R}^{d})$ algebraically. 
In $H_{2}$ he demands that $\mathcal{H}^{k}(\mathbb{R}^{d})$ is a locally convex Hausdorff space and that the inclusions 
$\mathcal{C}^{k}_{c}(\mathbb{R}^{d})\hookrightarrow\mathcal{H}^{k}(\mathbb{R}^{d})
\hookrightarrow\mathcal{C}^{k}(\mathbb{R}^{d})$ are continuous where $\mathcal{C}^{k}(\mathbb{R}^{d})$ 
has its usual topology and $\mathcal{C}^{k}_{c}(\mathbb{R}^{d})$ its inductive limit topology. 
In $H_{3}$ he supposes that a subset $B\subset\mathcal{H}^{k}(\mathbb{R}^{d})$ is bounded 
if and only if for every $\gamma\in\Gamma$ and $|\beta|\leq k$ the set 
$\{\gamma(x) \partial^{\beta}f(x)\;|\;x\in\mathbb{R}^{d},\, f\in B\}$ is bounded in $\mathbb{K}$. 
In $H_{4}$ he assumes that on every bounded subset of $\mathcal{H}^{k}(\mathbb{R}^{d})$ the topology of 
$\mathcal{H}^{k}(\mathbb{R}^{d})$ and the induced topology of $\mathcal{C}^{k}(\mathbb{R}^{d})$ coincide.

He defines the corresponding $E$-valued version $\mathcal{H}^{k}(\mathbb{R}^{d},E)$ of the space
$\mathcal{H}^{k}(\mathbb{R}^{d})$ for $\mathcal{H}^{k}=\mathcal{C}^{k}_{c}$, $\mathcal{C}^{k}$, 
$\mathcal{S}$ and $\mathcal{O}_{M}$ and shows that the statements of \prettyref{thm:weighted.diff} hold for 
all of them but $\mathcal{H}^{k}=\mathcal{C}^{k}_{c}$ (see \cite[p.\ 94-97]{Schwartz1955}, 
\cite[Proposition 9, p.\ 108]{Schwartz1955} and \cite[Th\'{e}or\`{e}me 1, p.\ 111]{Schwartz1955}).

In comparison, our conditions of local boundedness of $\mathcal{V}^{k}$ and being locally bounded away from zero on
$\Omega=\mathbb{R}^{d}$ imply $H_{1}$ and $H_{2}$ if the members of $\mathcal{V}^{k}$ are continuous. 
The assumption that the members of $\mathcal{V}^{k}$ are continuous is not a big difference 
if the members of the family $\mathcal{V}^{k}$ have a structure like in \eqref{eq:standard.weight}.
Then one may replace the indicator functions $\chi_{\Omega_{j}}$ by a smoothed version, 
e.g.\ by convolution of the indicator function with a suitable molifier, and then one gets a family of continuous 
weights which generates the same topology.
The condition $H_{3}$ is clearly fulfilled for the spaces $\mathcal{CV}^{k}(\mathbb{R}^{d})$ and 
the topology on them is called `topologie naturelle' by Schwartz (see \cite[p.\ 98]{Schwartz1955}). 
The condition $H_{4}$ implies that $\mathcal{C}^{k}_{c}(\mathbb{R}^{d},E)$ 
is dense in $\mathcal{H}^{k}(\mathbb{R}^{d},E)$ for $\mathcal{H}^{k}=\mathcal{C}^{k}$, 
$\mathcal{S}$ and $\mathcal{O}_{M}$ and quasi-complete $E$ 
(see \cite[p.\ 106]{Schwartz1955} and \cite[Th\'{e}or\`{e}me 1, p.\ 111]{Schwartz1955}). The same 
follows in our case from local boundedness and the cut-off criterion.
\end{rem}
\bibliography{biblio}

\begin{thebibliography}{10}

\bibitem{B3}
K.-D. Bierstedt.
\newblock {\em {Gewichtete R\"{a}ume stetiger vektorwertiger Funktionen und das
  injektive Tensorprodukt}}.
\newblock PhD thesis, Johannes-Gutenberg Universit\"{a}t Mainz, Mainz, 1971.

\bibitem{B1}
K.-D. Bierstedt.
\newblock {Gewichtete R\"{a}ume stetiger vektorwertiger Funktionen und das
  injektive Tensorprodukt. I}.
\newblock {\em J. Reine Angew. Math.}, 259:186--210, 1973.

\bibitem{B2}
K.-D. Bierstedt.
\newblock {Gewichtete R\"{a}ume stetiger vektorwertiger Funktionen und das
  injektive Tensorprodukt. II}.
\newblock {\em J. Reine Angew. Math.}, 260:133--146, 1973.

\bibitem{bourbakiII}
N.~Bourbaki.
\newblock {\em General Topology, Part 2}.
\newblock Elem. Math. Addison-Wesley, Reading, 1966.

\bibitem{buchwalter}
H.~{Buchwalter}.
\newblock {Topologies et compactologies.}
\newblock {\em {Publ. D\'ep. Math., Lyon}}, 6(2):1--74, 1969.

\bibitem{Defant}
A.~Defant and K.~Floret.
\newblock {\em Tensor norms and operator ideals}.
\newblock Math. Stud. 176. North-Holland, Amsterdam, 1993.

\bibitem{elstrodt2005}
J.~Elstrodt.
\newblock {\em Ma\ss- und Integrationstheorie}.
\newblock Grundwissen Mathematik. Springer, Berlin, 7th edition, 2011.

\bibitem{engelking}
R.~Engelking.
\newblock {\em General topology}.
\newblock Sigma Series Pure Math. 6. Heldermann, Berlin, 1989.

\bibitem{fabian}
M.~Fabian, P.~Habala, P.~H{\'a}jek, V.~Montesinos, and V.~Zizler.
\newblock {\em Banach Space Theory: The Basis for Linear and Nonlinear
  Analysis}.
\newblock CMS Books Math. Springer, New York, 2011.

\bibitem{F/W/Buch}
K.~Floret and J.~Wloka.
\newblock {\em Einf\"{u}hrung in die Theorie der lokalkonvexen R\"{a}ume}.
\newblock Lecture Notes in Math. 56. Springer, Berlin, 1968.

\bibitem{Gro}
A.~Grothendieck.
\newblock {\em Produits tensoriels topologiques et espaces nucl\'{e}aires}.
\newblock Mem. Amer. Math. Soc. 16. AMS, Providence, 4th edition, 1966.

\bibitem{H1}
L.~H\"{o}rmander.
\newblock {\em The Analysis of linear partial differential operators I}.
\newblock Classics Math. Springer, Berlin, 2nd edition, 1990.

\bibitem{james}
I.~M. James.
\newblock {\em Topologies and Uniformities}.
\newblock Springer Undergr. Math. Ser. Springer, London, 1999.

\bibitem{Jarchow}
H.~Jarchow.
\newblock {\em Locally Convex Spaces}.
\newblock Math. Leitf\"{a}den. Teubner, Stuttgart, 1981.

\bibitem{Kaballo}
W.~Kaballo.
\newblock {\em Aufbaukurs Funktionalanalysis und Operatortheorie}.
\newblock Springer, Berlin, 2014.

\bibitem{kriegl}
A.~Kriegl and P.~W. Michor.
\newblock {\em The Convenient Setting of Global Analysis}.
\newblock Math. Surveys Monogr. 53. AMS, Providence, 1997.

\bibitem{kruse2017}
K.~Kruse.
\newblock {Weighted vector-valued functions and the $\varepsilon$-product},
  2017.
\newblock arxiv preprint \url{https://arxiv.org/abs/1712.01613v6}.

\bibitem{meisevogt1997}
R.~Meise and D.~Vogt.
\newblock {\em Introduction to Functional Analysis}.
\newblock Clarendon Press, Oxford, 1997.

\bibitem{musial1987}
K.~Musia\l.
\newblock {Vitali and Lebesgue convergence theorems for Pettis integral in
  locally convex spaces}.
\newblock {\em Atti Sem. Mat. Fis. Univ. Modena}, 35:159--166, 1987.

\bibitem{nachbin1965}
L.~Nachbin.
\newblock {\em Elements of approximation theory}.
\newblock Notas de Mathem\'{a}tica 33. Instituto de Mathem\'{a}tica Pura e
  Aplicada, Rio de Janeiro, 1965.

\bibitem{Schwartz1955}
L.~Schwartz.
\newblock {Espaces de fonctions diff\'{e}rentiables \`{a} valeurs
  vectorielles}.
\newblock {\em J. Analyse Math.}, 4:88--148, 1955.

\bibitem{Sch1}
L.~Schwartz.
\newblock {Th\'{e}orie des distributions \`{a} valeurs vectorielles. I}.
\newblock {\em Ann. Inst. Fourier (Grenoble)}, 7:1--142, 1957.

\bibitem{Treves}
F.~Tr\`{e}ves.
\newblock {\em Topological Vector Spaces, Distributions and Kernels}.
\newblock Dover, New York, 2006.

\bibitem{Wilansky}
A.~Wilanksy.
\newblock {\em Modern Methods in Topological Vector Spaces}.
\newblock McGraw-Hill, New York, 1978.

\end{thebibliography}
\bibliographystyle{plain}
\end{document}